\documentclass[preprint,11pt]{elsarticle}

\usepackage{amsfonts, amsmath, amssymb}

\newtheorem{thm}{Theorem}[section]
\newtheorem{lem}[thm]{Lemma}
\newtheorem{prop}[thm]{Proposition}
\newtheorem{cor}[thm]{Corollary}

\newtheorem{rem}[thm]{Remark}

\newtheorem{exam}[thm]{Example}
\newtheorem{nota}[thm]{Notation}
\newtheorem{dfn}[thm]{Definition}
\newtheorem{conj}[thm]{Conjecture}

\newproof{proof}{Proof}

\newcommand{\bbC}{\mathbb C}
\newcommand{\bbN}{\mathbb N}
\newcommand{\bbR}{\mathbb R}
\newcommand{\bbZ}{\mathbb Z}

\newcommand{\bzero}{\mathbf 0}

\newcommand{\bm}{\mathbf m}
\newcommand{\br}{\mathbf r}
\newcommand{\bs}{\mathbf s}
\newcommand{\bt}{\mathbf t}
\newcommand{\bx}{\mathbf x}

\newcommand{\fg}{\mathfrak g}

\newcommand{\cG}{\mathcal G}

\newcommand{\cL}{\mathcal L}
\newcommand{\cW}{\mathcal W}

\newcommand{\rA}{\mathrm A}
\newcommand{\rE}{\mathrm E}

\newcommand{\rG}{\mathrm G}
\newcommand{\rK}{\mathrm K}
\newcommand{\rM}{\mathrm M}
\newcommand{\rN}{\mathrm N}
\newcommand{\rT}{\mathrm T}

\newcommand{\Ker}{\mathrm {Ker} \:}
\newcommand{\Supp}{\mathrm {Supp} \:}
\newcommand{\Span}{\mathrm {Span} \:}

\newcommand{\id}{\mathrm {id} \:}
\newcommand{\Ad}{\mathrm {Ad} \:}
\newcommand{\Forall}{\forall \:}

\begin{document}
%

\begin{frontmatter}
\title{On the full calculus of pseudo-differential operators on boundary groupoids with polynomial growth }
\author{Bing Kwan SO}
\ead{bkso@graduate.hku.hk}
\address{University of Luxembourg, \\
Campus Kirchberg, Mathematics Research Unit, BLG \\
6, rue Richard Coudenhove-Kalergi, \\
L-1359 Luxembourg, Grand Duchy of Luxembourg }


\begin{abstract}
In this paper, we enlarge the space of uniformly supported pseudo-differential operators on some groupoids
by considering kernels satisfying certain asymptotic estimates.
We show that such enlarged space contains the compact parametrix, 
and the generalized inverse of uniformly supported operators with Fredholm vector representation.
\end{abstract}


\end{frontmatter}

\maketitle

\section{Introduction}
In this paper, 
we estimate the kernel of generalized inverse of Fredholm Elliptic operators defined on boundary groupoids.

Our work is motivated by \cite{So;PhD}, 
which is in turn motivated by the study of differential operators on manifolds with boundary \cite{Melrose;Book},
\cite{Mazzeo;EdgeRev}.
 
Recall that in the classical construction, one first fixes a boundary defining function $\rho$,
a smooth non-negative function on $\rM$ with non-zero derivative on the boundary $\partial \rM$.
Then an open neighborhood of $\partial \rM \subset \rM$ is identified with $[0 , 1) \times \partial \rM$
(with $[0, 1) $ parameterized by $\rho$).
Differential operators tangential to the boundary are written in the form $\rho \partial _\rho + \cdots $,
and can be identified with kernels on the blowup $\rM _b $, known as the $b$-stretched product.
The $b$-stretched product has three boundary defining functions $\rho _{0 1} , \rho _{1 0} $ and $\rho _{1 1}$.
By some explicit calculations, 
it can be shown that the generalized inverse of a Fredholm elliptic operator is a kernel with asymptotic expansion in 
$\rho _{0 1} , \rho _{1 0} , \rho _{1 1}$.
The space of such kernels is known as the full calculus.

Following the classical theory, many variations emerge.
Most notable is the work of Gil, Krainer and Mendoza.
They considered `cone operators' of the form $ \rho ^{- m} \varPsi $, 
where $\varPsi $ is a $b$-differential operators as described above,
using somewhat similar techniques 
(see \cite{Krainer;ConeRes,Krainer;GeomSpecCone,Krainer;ConeRayMin,Gil;Adjoint}, 
however the theory of cone operators mainly concerns geodesically incomplete spaces and is therefore beyond our scope).

Closer to our discussion,
Lauter, Nistor and Monthubert studied the cusp, or $c _n$-calculus \cite{Nistor;Funct}.
They begin to use some elements of pseudo-differential operators on a groupoid,
which was first introduced by Nistor, Weinstein and Xu 
\cite{NWX;GroupoidPdO},
and further developed by Ammann, Lauter and Nistor into so called Lie manifolds,
or manifolds with Lie structure at infinity \cite{Nistor;LieMfld,Nistor;GeomOp}.
They prove that the Green function of elliptic operators has kernel that decays as a a Schwartz function.
However, these examples are quite similar to the manifold with boundary case,
and their argument make use of the explicit structure of the underlying groupoid,
described through the boundary defining function. 

Ammann, et. al. also apply similar theories to the example of polyhedral domains 
\cite{Nistor;Polyhedral,Nistor;Polyhedral2,Qiao;Thesis}.
In particular, 
\cite{Qiao;Thesis} considers inverse of differential operator as an element in the abstract 
$C^*$-algebra of the underlying groupoid.

The common theme of these results is that the Green function of elliptic differential operators
are in general not compactly support supported kernels.
One has to enlarge the calculus by considering non-compactly supported kernels of order $- \infty$,
possibly non-smooth ones.
Then one shows that the Green function lies in the enlarged calculus.

In the example of, say, natural differential operators on Poisson manifolds,
however, there is no obvious notion of boundary defining functions.
One can only use the theory of groupoid (pseudo)-differential operators to characterize these natural operators.

Motivated by the new class of examples, in \cite{So;PhD}, the author takes a more geometric approach.
The groupoid is taken as the fundamental object, 
and one attempts to do computations without explicitly referring the singular structure.
The idea was applied to the example of the symplectic groupoid of the Bruhat sphere,
where it was shown that the parametrix of an elliptic, uniformed supported pseudo-differential operator 
is given by a groupoid pseudo-differential with exponentially decaying kernel.

Our main objective is to generalize the result of \cite{So;PhD} to other similar groupoids, 
and also describe the generalized inverse of Fredholm operators.
As far as we know, 
this paper is the first systematic study on non-uniformed supported groupoid pseudo-differential operators
in some generality,
besides the purely abstract $C^*$-algebra construction in \cite{Nistor;GeomOp,Qiao;Thesis}.
Moreover, our work should clarify the role of the boundary defining function in these works, 
as well as the classical construction. 

\subsection{An overview of our approach}
While the technical details are tedious and elementary, 
the idea behind our construction is actually very simple.

In Section 2, we recall some basic notions of pseudo-differential operators on a groupoid as in \cite{NWX;GroupoidPdO}.
Then we define the notion of boundary groupoids. 
Essentially these groupoids are just $b$-stretched products with possibly non-commutative isotropy subgroups
and more degenerate Lie algebroids.

In Section 3, we begin with an elementary estimate. 
Perhaps what is remarkable is that such estimate has no direct analogue in the classical construction.
Then we write down the definition of the calculus with bounds. 
These are just kernels that decays exponentially on the $\bs$-fiber and polynomially near the singular set,
with respect to some rather arbitrarily chosen functions.
We show that the convolution product respects the filtration of the calculus with bounds.

In Section 4, we describe the generalized inverse of elliptic differential operators 
(or uniformly supported pseudo-differential operators). 
Our construction is parallel to that of \cite{Mazzeo;EdgeRev}.

Given an elliptic, uniformly supported pseudo-differential operator $\varPsi = \{ \varPsi _x \} _{x \in \rM }$
such that the vector representation of $\varPsi $ is Fredholm,
one starts with the invariant sub-manifold $\cG _r $ with the lowest dimension.
In that case $\varPsi |_{\cG _r } $ 
is an ordinary pseudo-differential operator on the manifold with bounded geometry $\rM _r \times \rG _r$
that is invertible.
Therefore the result of Shubin \cite{Shubin;BdGeom} applies and  
$\varPsi |_{\cG _r } ^{-1 }$ is a kernel with exponential decay.

The second step is to extend the off-diagonal part of $\varPsi |_{\cG _r } ^{-1 }$ into a kernel on $\cG$.
In the case $\cG = \rM _0 \times \rM _0 \bigsqcup \rG \times \rM _1 \times \rM _1 $,
this is constructed by taking exponential coordinates patches defined by Nistor et. al. \cite{Nistor;IntAlg'oid}
and then extend along coordinates curves.
The detail of the construction is given in Appendix B.
Then a uniformly supported parametrix $\varPhi$ of $\varPsi $ on $\cG$ can be modified,
so that $R := I - \varPsi \varPhi $ vanishes on $\cG _r$.

The third step is to improve the parametrix by considering the Neumann series.
One gets a parametrix up to error decaying at order $- \infty $ at the singular set.
The same arguments can be repeated and the case for general $r $ can be proved by induction on $r$.
In the last step of the induction, one obtains the generalized inverse.

In Section 5, we give some more remarks and highlight some open problems.

\subsection{Acknowledgment}
Part of the paper is based on, and much more inspired by, my PhD research at Warwick University. 
I would like to thank Shantanu Dave, Victor Nistor, Yu Qiao, John Rawnsley, Ping Xu,
and the anonymous reviewer for many useful discussions,
and the hospitality of The Penn State University during my visit at summer 2011. 
This project is supported by the AFR (Luxembourg) postdoctoral fellowship.

\section{Preliminary definitions}

\subsection{Uniformly supported pseudo-differential calculus on a Lie groupoid}
In this section, we recall the standard theory of pseudo-differential calculus developed by
Nistor, Weinstein and Xu \cite{NWX;GroupoidPdO}.

\begin{dfn}
A {\it Lie groupoid} $\cG \rightrightarrows \rM $ consists of:
\begin{enumerate}
\item Manifolds $\cG$ and $\rM$;
\item A unit inclusion ${\mathbf u} : \rM \rightarrow \cG$;
\item Submersions $\bs , \bt : \cG \rightarrow \rM$, called the source and target map respectively,
satisfying
$$ \bs \circ {\mathbf u} = \id _\rM = \bt \circ {\mathbf u}; $$
\item A multiplication map 
$\mathbf m : \{ (a, b) \in \cG \times \cG : \bs (a) = \bt (b) \} \rightarrow \cG,
(a , b) \mapsto a b$
that is associative and satisfies
$$ \bs (a b) = \bs (b) , \quad \bt (a b) = \bt (a), 
\quad a ( \mathbf u \circ \bs (a)) = a = ( \mathbf u  \circ \bt (a)) a ; $$
\item An inverse diffeomorphism $\mathbf i : \cG \rightarrow \cG, a \mapsto a^{-1}$,
such that $\bs (a^{-1}) = \bt (a), \\ \bt (a^{-1}) = \bs (a)$ and 
$$ a a^{-1} = \mathbf u (\bt (a)), a^{-1} a = \mathbf u (\bs (a)).$$
\end{enumerate} 
\end{dfn}
Our definition follows the convention of \cite{Mackenzie;Book2}, 
but with the source and target maps denoted by $\bs$ and $\bt$ instead of $\alpha $ and $\beta $.

Let $\cG \rightrightarrows \rM $ be a Lie groupoid with $\rM$ compact.
Fix a metric $g _\rA $ on $\rA$. 
For each $x \in \rM $, a Riemannian metric on $\cG _x $ is defined by right invariance.
Denote the family of Riemannian volume measure on $\cG _x , x \in \rM$ by $\mu _x $.

Observe that for each $x \in \rM $, $\cG _x $ is a manifold with bounded geometry (see Appendix A).
Therefore, $\cG _x $ has at most exponential volume growth.

\begin{dfn} 
We say that $\cG$ is of polynomial (volume) growth if there exists $N \in \mathbb N , C > 0 $ such that
$$ \int _{B _{g _\rA } (a , r) } \mu _x (b) \leq C r ^N ,$$
for any ball on $\cG _x $ centered at $a \in \cG _x $ with radius $r$.
\end{dfn}

From now on, we shall always assume that the groupoid $\cG$ under consideration is of polynomial growth. 

\begin{dfn}
A pseudo-differential operator $\varPsi $ on a groupoid $\cG$ of order $\leq m$ 
is a smooth family of pseudo-differential operators $\{ \varPsi _x \}_{x \in \rM}$,
where $\varPsi _x \in \Psi ^m ( \cG _{x} )$,
and satisfies the right invariance property
$$ \varPsi _{\bs (a)} (\br _a^* f) = \br _a^* \varPsi _{\bt (a)} (f), 
\quad \Forall a \in \cG, f \in C^\infty_c (\cG _{\bs (a)}).$$
If, in addition, all $\varPsi _x $ are classical of order $m$, then we say that $\varPsi $ is classical of order $m$.
\end{dfn}

For a pseudo-differential operator $\varPsi = \{ \varPsi _x \}$ on $\cG$.
The support of $\varPsi $ is defined to be 
$$ \Supp (\varPsi) = \overline {\bigcup_{x \in \rM} \Supp (\varPsi _x)}.$$
The operator $\varPsi $ is called uniformly supported if the set
$$ \{ a b^{-1} : (a, b) \in \Supp (\varPsi) \} $$
is a compact subset of $\cG$. 
We denote the algebra of uniformly supported classical pseudo-differential operator of order $m$ on $\cG $ by 
$\Psi ^{[m]} (\cG)$ ($\Psi ^{[m]} (\cG, \rE)$ for operators defined on a sections of a vector bundle $\rE \to \rM$),
and $\Psi ^\bullet := \bigcup _{m \in \mathbb Z } (\cG )$. 

The convolution product on $\cG$ is the binary operator on $C ^\infty (\cG)$:
\begin{equation}
f \circ g (a) := \int _{\bs ^{-1} (\bs (a))} f (a b ^{-1} ) g (b) \mu _{\bs (a)} (b), \quad \Forall a \in \cG 
\end{equation}
for any $f, g \in C ^\infty (\cG)$, provided the integral is finite for all $x \in \rM$.

For any $\varPsi = \{ \varPsi _x \}_{x \in \rM} \in \Psi ^\infty (\cG)$.
The {\it reduced kernel} of $\varPsi $ is defined to be the distribution
$$ \psi (f) :=  \int _\rM \mathbf u^* (\varPsi (\mathbf i^* f)) (x) \: \mu _\rM (x), 
\quad f \in C^\infty _c (\cG).$$

\begin{lem}
\label{RedKerId}
\cite[Corollary 1]{NWX;GroupoidPdO}
For any $\varPsi \in \Psi ^\bullet (\cG)$, the reduced kernel is co-normal at $\rM$ and smooth elsewhere.
Moreover, the map $\varPsi \mapsto \psi $, where $\psi $ is the reduced kernel of $\varPsi $, 
is an algebra isomorphism.
\end{lem}

\begin{rem}
By virtue of Lemma \ref{RedKerId}, 
there are three equivalent ways to define the algebra of pseudo-differential operator on $\cG$, namely
\begin{enumerate}
\item
Fiberwise composition $ \varPsi \varPhi = \{ \varPsi _x \varPhi _x \}_{x \in \rM }$;
\item
Convolution product $\psi \circ \varphi $, 
where $\psi $ and $\varphi $ is the reduced kernel of $\varPsi $ and $\varPhi $ respectively;
\item
The fiberwise operation $ \varPsi _x (\varphi |_{\cG _x}) , x \in \rM $.
\end{enumerate}
\end{rem}

For any $\varPsi \in \Psi ^\bullet (\cG )$, the vector representation of $\varPsi $ is the operator
$\nu (\varPsi ) : C ^\infty (\rM ) \to C ^\infty (\rM )$,
$$ (\nu (\varPsi ) f) := \varPsi _x (\bt ^* f)|_{\rM }. $$ 
Note that if $X \in \Gamma ^\infty (\rA)$ is regarded as a differential operator on $\cG$,
then the vector representation of $X$ is just $\nu (X)$,
the image of $X$ under the anchor map (regarded as a differential operator on $\rM$),
so there is no confusion using the same notation for both.

\subsection{Boundary groupoids}
We define the main object we are interested in.
\begin{dfn} 
Let $\cG \rightrightarrows \rM $ be a Lie groupoid with $\rM$ compact.
We say that $\cG$ is a {\it boundary groupoid} if 
\begin{enumerate}
\item
The singular foliation defined by anchor map $\nu : \rA \to T \rM $ has a finite number of leaves 
$\rM _0 , \rM _1 , \cdots , \rM _r \subset \rM $, called invariant sub-manifolds on $\rM$,
such that $\dim \rM = \dim \rM _0 > \dim \rM _1 > \cdots > \dim \rM _r $;
\item
For all $k = 0, 1, \cdots r$, 
$\bar \rM _k := \rM _ k \bigcup \rM _{k + 1} \bigcup \cdots \bigcup \rM _r $ are closed, 
immersed sub-manifolds of $\rM$;
\item
$\cG _{0 } := \bs ^{-1} (\rM _0 ) = \bt ^{-1} (\rM _0 ) \cong \rM _0 \times \rM _0$, the pair groupoid, and 
$\cG _{k } := \bs ^{-1} (\rM _k ) = \bt ^{-1} (\rM _k )
\cong \rG _k \times ( \rM _k \times \rM _k ) $ for some Lie groups $\rG _k $;
\item
For each $k$, there exists (unique) sub-bundles $\bar \rA _k \subset \rA |_{\bar \rM _k}$
such that $\bar \rA |_{\rM _k } = \ker (\nu |_{\rM _k }).$  
\end{enumerate}
\end{dfn}
For simplicity, we shall also assume that $\rG _k $ and $\rM _k $ are connected, hence all $\bs$-fibers are connected.

\begin{rem}
Note that we do not need any condition on the dimension on $\rM _k$.
On the other hand since for each $k = 0, \cdots , r - 1$, $\rM _k$ is an open and dense subset of $\bar \rM _k$,
it follows $\bar \rM _{k +1 } $ is the boundary of $\rM _k$ in $\rM$.
\end{rem}

\begin{nota}
We shall also denote $\bar \cG _k := \bs ^{-1} (\bar \rM _k ) = \bt ^{-1} (\bar \rM _k ) $.
\end{nota}  

\begin{exam}
Let $\rM = \rM _0 \bigsqcup \rM _1 $ be a manifold with embedded boundary \cite{Melrose;Book}.
The groupoid of space of totally characteristic operators is given (as a set) by
$$ \cG := (\rM _0 \times \rM _0 ) \bigsqcup \bbR \times (\rM _1 \times \rM _1 ).$$
Note that $\cG $ is an open dense subset of the blowup of $\rM$, 
known as the b-stretched product (see \cite{M'bert;CornerGroupoids}).
\end{exam}

\begin{exam} 
\label{BruhatExam}
\cite[Example 2.18]{So;PhD} (See also \cite{Lu;PoissonCohNotes})
Let $\rK = \mathrm {S U } (n) $, $\rT \subset \rK$ be the maximal torus, 
$\rN $ be the Lie group of upper triangular matrices with unit diagonal.

Define the left action of $\rT$ on $\rK \times \rN$ by
$$ g \cdot (k , n) := (g k, g n g ^{-1}), \quad \Forall (k, n) \in \rK \times \rN, g \in \rT .$$
It is easy to see that the projection onto $ \rT \backslash (\rK \times \rN) $ is a submersion. 

Define the groupoid operations on $\cG := \rT \backslash (\rK \times \rN) \rightrightarrows \rT \backslash \rK $:
\begin{align*}
\text {source and target: } & \bs ( {}_\rT (k, n) ) = {}_ \rT k , \bt ( {}_\rT (k, n) ) := {}_\rT k' , \\
& \text {where } n k = k' a' n' \text { is the (unique) Iwasawa decomposition;} \\  
\text {multiplication: } &  \mathbf m ({}_\rT (k _1 , n _1) , {}_ \rT (k _2 , n _2) ) := {}_\rT (k _2 , n _1 n _2) , \\
& \text {provided one has Iwasawa decomposition } n _2 k _2 = k _1 a' n' ; \\
\text {inverse: } & \mathbf i ( {}_\rT (k, n)) := {}_\rT (k', n^{-1}) , \\
& \text {where } n k = k' a' n' \text { is the (unique) Iwasawa decomposition;} \\  
\text {unit: } & \mathbf u ({}_ \rT k ) := {}_ \rT (k , e) , e \in \rN .
\end{align*} 

Remark that $\rT \backslash (\rK \times \rN) \rightrightarrows \rT \backslash \rK $ 
is just the symplectic groupoid of the Bruhat Poisson structure on $\rK$.
In particular, when $n = 2$, $ \rT \backslash \rK $ is just the sphere $\mathbb S ^2$.
Let $(x, y) $ be the stereographic coordinate opposite to ${}_\rT e$,
the Poisson bi-vector field is 
$$ \varPi = (x ^2 + y ^2 )(1 + x ^2 + y ^2 ) \partial _x \wedge \partial _y .$$
The Lie algebroid is $\rA = T ^* \mathbb S ^2 $ and the anchor map is contraction with $\varPi $.
\end{exam}

\begin{exam}
Recall that any Poisson structure $\varPi \in \Gamma ^\infty (\wedge ^2 T \rM ) $ defines a Lie algebroid structure
on the cotangent bundle $T ^* \rM$, 
and any bi-vector field $\varPi $ on a two dimensional manifold $\rM$ is Poisson (see, for example, \cite{Vas;Book}).
In particular, let
$\varPi $ be any bi-vector field on the sphere $\mathbb S ^2$ such that $\varPi $ vanishes at exactly one point $x _0 $.
Then the Lie algebroid is integrable \cite{Debord;IntAlgebroid}.
It is easy to see that the groupoid integrating $T ^* \mathbb S ^2$ must be a boundary groupoid.
Since on some fixed local coordinates around $x _0 $ one can take 
$\varPi = f (x, y) \partial _{x } \wedge \partial _{y} $ for arbitrary smooth function $f $ vanishing at $x _0 $ only,
differential operators obtained this way in general cannot be reduced to the cases considered in \cite{Mazzeo;EdgeRev}. 
\end{exam}

\subsection{The Fredholmness criterion of Lauter and Nistor}
By definition, any boundary groupoids 
$\cG \rightrightarrows \rM $ satisfy the condition of Lauter and Nistor \cite{Nistor;GeomOp},
which we recall here:
\begin{dfn}
\label{BdGpoid}
An $\bs$-connected groupoid $\cG \rightrightarrows \rM$ is said to be a {\it Lauter-Nistor groupoid} if
\begin{enumerate}
\item
The unit set $\rM $ is compact;
\item
The anchor map $\nu : \rA \rightarrow T \rM$ is bijective over an open dense subset $\rM_0 \subseteq \rM$;
\item
The Riemannian manifold $(\rM_0 , g _{\rM _0}) $ has positive injectivity radius
and has finitely many connected components $\rM _0 = \coprod _\alpha \rM _\alpha $;
\item
As a groupoid,
$\cG_{\rM_0} \cong \coprod _\alpha \rM_\alpha \times \rM_\alpha $, the pair groupoid.
\end{enumerate}
\end{dfn}

Let $\varPsi $ be pseudo-differential operator on $\cG $.
By right invariance, it is clear that 
\begin{lem}
For any $x \in \rM $, $\varPsi _x $ is a uniformly bounded pseudo-differential operator 
on the manifold with bounded geometry $\cG _x $.
\end{lem}

Moreover, since $\cG _{\rM _0 } \cong \rM _0 \times \rM _0 $,
we define a Riemannian metric on $\rM _0 \subseteq \rM $ by taking the metric on $\cG _x $ for any $x \in \rM _0 $. 
Now since $\rM _0 $ is a manifold with bounded geometry, 
we shall consider the `natural' Sobolev spaces $\cL ^2 (\rM _0 ) $ 
and $\cW ^m (\rM _0 ) $ as defined in Equation (\ref{UBSobo}) in the appendix.

Recall that any uniformly bounded pseudo-differential operator of order $m$ on the manifolds with bounded geometry
$\cG _x $ (and $\rM _0 $) extends to a bounded linear map from 
$\cW ^m (\cG _x , \bt ^{-1} \rE ) $ to $ \cL ^2 (\cG _x , \bt ^{-1} \rE )$.
With these notations, the Fredholmness criterion of Lauter and Nistor can be stated as:
\begin{prop}
\cite[Theorem 9]{Nistor;GeomOp} 
Let $\cG $ be a Lauter-Nistor groupoid and $\varPsi \in \Psi ^{[m]} (\cG , \rE)$ be elliptic. 
Then $\nu (\varPsi) : \cW ^m (\rM _0 , \rE ) \to \cL ^2 (\rM _0 , \rE )$ is Fredholm if and only if,
for all $x \in \rM \setminus \rM _0$,
$\varPsi _x : \cW ^m (\cG _x , \bt ^{-1} \rE ) \to \cL ^2 (\cG _x , \bt ^{-1} \rE )$ are invertible.
\end{prop}

\section{The calculus with bounds}

\subsection{A structural theorem for Lie algebroids}
Given a boundary groupoid $\cG \rightrightarrows \rM$.
Fix any Riemannian metric $\bar g $ on $\rM $.
For each $k \geq 1 $, let $d ( \cdot , \bar \rM _{k }) $ be the distance function defined by $ \bar g $.

For each $k \geq 0$, fix a function $\rho _k \in C^\infty (\rM)$ such that 
$\rho _k > 0 $ on $\rM \setminus \bar \rM _k $ 
and $\rho _k = d ( \cdot , \bar \rM _k) $ on some open set containing $ \bar M _k $.

Fix a metric $g _\rA $ on $\rA $ (i.e. a positive bi-linear form on $\rA$).
Then $ g _\rA $ induces Riemannian metrics on the 
$\bs $-fibers $\cG _ x := \bs ^{-1} (x) $ for each $x \in \rM $ by right invariance.

\begin{exam}
To motivate the construction, we briefly consider the example of manifold with embedded boundary \cite{Melrose;Book}.
Write $\rM = \rM _0 \bigsqcup \rM _1 $,
Take a collar neighborhood of the boundary $\rM _1 $, $[0 , 1) \times \rM _1 \subset \rM $.
For simplicity, assume $\bar g $ is such that $\bar g |_{ [0 , 1) \times \rM _1 }$ is the product metric. 
Then $\rho _1 := d _{\bar g} ( \cdot , \rM _1 ) $ is just the boundary defining function.
\end{exam}

We shall see in this section how the $\rho _k$'s play the role of boundary defining function.
We begin with considering the bundle map $ d \rho _k \circ \nu : \rA \to \bbR \times \rM $.

\begin{lem} 
\label{LocalDegen}
For each $k$, there exists a constant $\omega _k$ 
such that for any $ x $ lying in some open neighborhood of $\bar \rM _k$, $ X \in A _x ,$
\begin{equation}
\label{LocalDegEq}
| d \rho _k \circ \nu (X) | \leq \omega _k \rho _k (x) | X |_{g _\rA} .
\end{equation}
\end{lem} 
\begin{proof} 
Let $x _0 \in \bar \rM _k $ be arbitrary.
Fix a coordinate chart $(\bar U , \bar \bx )$ around $x$.
We may assume that $ T \bar \rM _k ^\bot \subset T \rM $ is trivial on $\bar U$.
Then the map 
$$(x _1 \cdots x _n ) \mapsto \exp _{\bar \bx (x _1 , \cdots x _m )} (x_{m+1} , \cdots , x _n ) ,$$
where $\exp $ is the exponential map defined by the Riemannian metric $\bar g $,
defines a set of local coordinates on some open subset $U \subset \rM$.
Moreover, by definition of the exponential map, one has
$$ d \rho _k (\partial _i ) = 0 , \quad \Forall i \leq m .$$  

On the other hand, 
we may assume that $\rA $ is trivial on $U$ and fix any orthonormal basic sections
$E _1 , \cdots , E _n \in \Gamma ^\infty (\rA )$ and write
\begin{equation}
\nu (E _j ) = \sum_{i = 1 } ^n \nu _{i j } \partial _i
\end{equation}
for some smooth functions $ \nu _{i j}$.
Compositing with $ d \rho _k$, one gets 
$$ d \rho _k \circ \nu (E _j ) = \sum_{i = m+1} ^n \nu _{i j } (\partial _i \rho _k ).$$
Since the image of $\rA |_{\bar M _k }$ under $\nu $ lies in $T \bar \rM _k $, it follows that
$$ \nu _{i j} ( x ) = 0 , \quad \Forall i > m \text{ and } x \in \bar U \subset \bar M _k.$$
The smoothness of $\nu _{i j} $ implies there exists an open subset $U' \subset U$ containing $x _0 $,
and constant $\omega _{k , U}$ such that
$$ | \nu _{i j } (x)| \leq \omega _{k , U} \rho _ k (x) , \quad \Forall x \in U'.$$
Since $x _0 $ is arbitrary, the lemma follows by considering a suitable finite cover of $\bar \rM _k$. 
\end{proof}

\begin{rem}
\label{ModNote}
Given $\omega _k $ as in Lemma \ref{LocalDegen}, 
we may modifying $\rho _k $ outside a neighborhood of $\rM _k $ to get
$$ | d \rho _k \circ \nu (X) | \leq \omega _k \rho _k (x) | X |_{g _\rA}  $$
Since we shall only be interested in estimates up to some multiples, 
it is clear that such modification have no effect on the arguments.
Therefore we shall often implicitly assume such modification is being made if necessary.  
\end{rem}

\begin{thm} 
\label{DEst}
For each $k$, let $\omega _k$ be defined in the previous Lemma \ref{LocalDegen}. 
Suppose further that $ | d \rho _k \circ \nu (X) | \leq \omega _k \rho _k (x) | X |_{g _\rA} $ for any $X \in \rA $.
Then for any $x \in \rM , a, b \in \cG _x $, 
\begin{equation}
\omega _k d (a, b ) \geq \left| \log \left( \frac{\rho _k (\bt (b))}{\rho _k (\bt (a))} \right) \right| .
\end{equation}
\end{thm}
\begin{proof} 
Given any $a, b \in \cG _x $, since $\cG _x$ is a complete, connected Riemannian manifold, 
there exists a minimizing geodesic $\gamma : [0,1] \to \cG _x $ connecting $a$ and $b$.
Define the curves $\gamma _\rA (t) := d R _{(\gamma (t)) ^{-1}} (\gamma ' (t)) \in \rA$,
and $\gamma _\rM (t) := \bt \circ \gamma (t) \in \rM$.
Then one has the relation
$$ \nu (\gamma _\rA (t)) = \gamma '_\rM (t) .$$

Moreover, by right invariance, it follows that
$$ d (a , b ) = \text {length of $\gamma $} = \int _0 ^1 \vert \gamma _\rA (t) \vert _{g _\rA } d t .$$
Applying Lemma \ref{LocalDegen}, one gets
\begin{align*}
\omega _k d (a, b) &=  \int _0 ^1 C \vert \gamma _\rA (t) \vert _{g _\rA } d t \\
& \geq \int _0 ^1 \frac {\vert d \rho _k (\nu (\gamma _\rA (t))) \vert }{\rho _k (\gamma _\rM (t))} d t 
\geq \left| \int _0 ^1 \frac {d \rho _k ( \gamma '_\rM (t)) }{\rho _k (\gamma _\rM (t))} d t \right| 
= \left| \log \left( \frac{\rho _k (\bt (b))}{\rho _k (\bt (a))} \right) \right| .
\end{align*}
\end{proof}

\begin{rem}
Here, we observe that on can instead take any non-negative functions $\tilde \rho _k \in C ^0 (\rM )$, such that
\begin{enumerate}
\item
$\tilde \rho _k = 0 $ on $\bar \rM _k$, $\tilde \rho _k $smooth and positive on $\rM \setminus \rM _k $;
\item
Theorem \ref{DEst} holds for $\tilde \rho _k $ for some $\omega _k > 0$;
\item
One has $M \rho _k ^\lambda \leq \tilde \rho _k \leq M' \rho _k {\lambda '} $ 
for some $M, M' , \lambda , \lambda ' >0 $,
\end{enumerate}
and all the subsequent arguments remain true. 
At this point it is unclear if there is an `optimal' choice for the defining functions $\rho _k$.
\end{rem}

Inspired by Lemma \ref{LocalDegen} and Theorem \ref{DEst}, we define:
\begin{dfn}
The groupoid $\cG$ is said to be uniformly non-degenerate if there exist constants 
$\omega'_1 , \omega ' _2 , \cdots , \omega ' _r > 0 $ such that 
$$ |\nu (X) | \geq \omega ' _k \rho _k (x) | X |_{g _\rA} ,
\quad \forall \: x \in \bar \rM _{k - 1}, X \in A _x , X \bot \bar \rA _k ;$$
The groupoid $\cG$ is said to be uniformly degenerate if there exist constants 
$\omega _1 , \omega_2, \cdots , \\ \omega _r , \omega '_1 , \cdots \omega '_r > 0 $
and exponents $\lambda _1 , \cdots \lambda _r , \lambda' _1 , \cdots , \lambda ' _r \geq 2$ such that 
\begin{align*}
| d \rho _k \circ \nu (X) | & \leq \omega _k (\rho _k (x))^{\lambda _k} | X |_{g _\rA},
\quad \forall \: x \in \bar \rM _{k - 1}, X \in A _x , \text { and } \\
|\nu (X) | & \geq \omega ' _k (\rho _k (x))^{\lambda ' _k } | X |_{g _\rA} ,
\quad \forall \: x \in \bar \rM _{k - 1}, X \in A _x , X \bot \bar \rA _k .
\end{align*}
\end{dfn}

\begin{rem}
The groupoid $\cG$ is uniformly degenerate if and only if
$$ d \nu _{i j} ( x ) = 0 , \quad \Forall i \leq m \text{ and } x \in \bar U \supset \bar M _k,$$
where $\nu _{i j} $ is defined in Equation (2).
\end{rem}

\begin{rem}
If $\cG $ is uniformly non-degenerate, it is necessary that $\dim M _k = \dim M - k $.
\end{rem}

Applying the same arguments as in Lemma \ref{LocalDegen} and Theorem \ref{DEst}, it is obvious that:
\begin{cor}
If $\cG $ is totally degenerate, then for any $\omega_k > 0 $, 
there exists a function $\rho ' _k $ such that $\rho ' _k = \rho _k $ on some open neighborhood of $\rM _k $
(which depends on $\omega _k $), 
and satisfies
$$ \omega _k d (a, b ) \geq \left| \log \left( \frac{\rho ' _k (\bt (b))}{\rho ' _k (\bt (a))} \right) \right| ,$$
for all $a ,b \in \cG$ such that $\bs (a) = \bs (b)$.
\end{cor}

\subsection{Construction of the calculus with bounds}
Given a Hausdorff groupoid $\cG$, 
defined the manifold $ \tilde \cG := \{ (a , b ) \in \cG \times \cG : \bs (a ) = \bs (b) \} $.
Let $\tilde \bm $ denote the map 
$$(a , b ) \mapsto a b ^{-1}, \quad (a, b ) \in \tilde \cG .$$ 

Also recall that for any $X \in \Gamma ^\infty (\rA )$,
$X$ determines a right invariant vector field $X ^\br \in \Gamma ^\infty (\Ker (d \bs ))$.
Moreover, given vector fields on $X ^\br , Y ^\br $ on $\cG $, 
$$ X ^\br (a) \oplus Y ^\br (b) \in T _{(a, b)} \tilde \cG \subset T _{(a , b)} (\cG \times \cG ) $$
for any $( a , b ) \in \tilde \cG $.
We shall consider $X ^\br \oplus Y ^\br $ as a vector field on $\tilde \cG$.

Consider $d \tilde \bm (X ^\br \oplus Y ^\br )$.
Observe that $X ^\br $ is just the vector field
$$ X ^\br (a) = \partial _t |_{t = 0 } \exp t X (\bt (a)) (a) , \quad \Forall a \in \cG .$$ 
It follows that for any $X , Y \in \Gamma ^\infty (\rA ), ( a , b ) \in \tilde \cG $, 
\begin{align} 
\label{ReducedVF}
d \tilde \bm (X ^\br \oplus Y ^\br ) (a b ^{-1}) 
&= \partial _t |_{t = 0} ( \exp t X (\bt (a))) a b ^{-1} (\exp t Y (\bt (b)))^{-1} \\ \nonumber
&= \partial _t |_{t = 0} ( \exp t X (\bt (c))) c (\exp t Y (\bt (c)))^{-1} ,
\end{align}
where $c := a b ^{-1}$.
In particular, $d \tilde \bm (X ^\br \oplus Y ^\br ) $ is a well defined vector field on $\cG $.

Recall how the $\bs$-fiberwise covariant derivatives of a function is defined.
Let $\nabla $ be the Levi-Civita $\rA $-connection with respect to the given Riemannian metric $g _\rA$. 
By right invariance, $\nabla $ defines the Levi-Civita connection on each $\bs$-fiber $\cG _x$,
which we still denote by $\nabla$.
For any smooth functions $\psi $ on $\cG$, 
$\nabla ^l \psi \in \Gamma ^\infty (\otimes ^l \ker (d \bs )) , l = 1, 2, \cdots ,$ is defined inductively by
\begin{align}
\label{HigherD}
\nabla \psi (X _1 ^\br ) := & L _{X _1 ^\br } \psi \\ \nonumber
\nabla ^l \psi (X _1 ^\br , X _2 ^\br , \cdots , X _l ^\br ) 
:= & L _{X _1 ^\br } (\nabla ^{l - 1} \psi (X _2 ^\br , \cdots , X _l ^\br )) \\ \nonumber
&- \sum _{k = 2} ^l \nabla ^{l - 1} \psi (X _2 ^\br , \cdots , (\nabla _{X _1} X _k )^\br , \cdots X _l ^\br ).
\end{align}
Likewise, on $\tilde \cG $, let $\tilde \nabla $ be the Cartesian product connection.
One considers higher covariant derivatives $\tilde \nabla ^l $.
In particular, observe that for any $\psi \in C ^\infty (\cG )$, 
\begin{align}
\tilde \nabla (\tilde \bm ^* \psi ) (V _1 ^\br \oplus W _1 ^\br ) (a, b) 
=& L _{d \tilde \bm (V _1 ^\br \oplus W _1 ^\br )} (a b ^{-1}) \\ \nonumber
\tilde \nabla ^2 (\tilde \bm ^* \psi ) ( V _1 ^\br \oplus W _1 ^\br , V _2 ^\br \oplus W _2 ^\br ) (a , b) 
=& L _{ d \tilde \bm ( V _1 ^\br \oplus W _1 ^\br )} L _{ d \tilde \bm ( V _2 ^\br \oplus W _2 ^\br )} 
\psi (a b ^{-1}) \\ \nonumber
&- L _{ d \tilde \bm ( \nabla _{V _1} V _2 ^\br \oplus \nabla _ {W _1} W _2 ^\br )} \psi (a b ^{-1}),
\end{align}
and so on.
 
For each $( a , b ) \in \tilde \cG $, 
define $d (a , b) $ to be the Riemannian distance on $\cG _{\bs (a)} = \cG _{\bs (b)}$ between $a$ and $b$.

\begin{dfn}
For each $\varepsilon > 0 $, 
define the {\it exponentially decaying calculus of order $\varepsilon $} to be the space of kernels
\begin{align*}
\Psi ^{- \infty } _{\varepsilon ; \bzero } (\cG) :=
\Big\{ \psi \in C ^0 ( \cG ) : \psi |_{\cG _k } \in C ^\infty ( \cG _k) &,
\Forall l \in \bbN, \exists M _l > 0 \text { such that } \\
\text{ for any } X _1 , \cdots , X _l , Y _1 &, \cdots Y _l \in \Gamma ^\infty (\rA ), (a , b) \in \tilde \cG , \\
\tilde \nabla ^l (\bm ^* \psi ) (X _1 ^\br \oplus Y _1 ^\br & , 
\cdots , X _l ^\br \oplus Y _l ^\br  )(a , b) \\
\leq M _l e ^{- \varepsilon ' d (a , b)} & (|X_1 | + |Y_1|)(|X_2 | + |Y_2|) \cdots (|X_l | + |Y_l|) \Big\}.
\end{align*}
\end{dfn}

\begin{rem}
For simplicity we only consider the scalar case. 
A groupoid pseudo-differential operators on a vector bundle $\mathrm E \to \rM $ can be identified with a 
(distributional) section on $\bt ^{-1} \mathrm E \otimes \bs ^{-1} \mathrm E \to \cG$.
One instead considers covariant derivative on $\mathrm E$ and it is clear that all arguments below follows.
\end{rem}

Suppose that $\cG$ is a groupoid of polynomial growth. 
Then for any $\varepsilon _1 , \varepsilon _2 \geq 0 $,
$\psi _1 \in \Psi ^{- \infty } _{\varepsilon _1 ; \bzero } (\cG ), 
\psi_2 \in \Psi ^{- \infty } _{\varepsilon _2 ; \bzero} (\cG )$,
the convolution product $ \psi _1 \circ \psi _2 $ is well defined.
Moreover, as in the case of manifolds with boundary \cite{Mazzeo;EdgeRev, Melrose;Book}, 
one has
\begin{lem}
\label{CompoLem}
Let $\cG$ be a groupoid of polynomial growth.
For any $\varepsilon _1 , \varepsilon _2 \geq 0 $
$$ \Psi ^{- \infty } _{\varepsilon _1 ; \bzero } \circ \Psi ^{- \infty } _{\varepsilon _2 ; \bzero} 
\subseteq \Psi ^{- \infty } _{\min \{ \varepsilon _1 , \varepsilon _2 \}}.$$ 
\end{lem}
\begin{proof}
For simplicity we only consider the scalar case.
It suffices to consider the convolution product
$ \psi _1 \circ \psi _2 $
for any $\psi _1 \in \Psi ^{- \infty } _{\varepsilon _1 ; \bzero } (\cG ), 
\psi_2 \in \Psi ^{- \infty } _{\varepsilon _2 ; \bzero} (\cG )$.
In view of the formula
$$ \psi _1 \circ \psi _2 (a) = \int _{b \in \cG _{\bs (a)}} \psi _1 (a b ^{-1}) \psi _2 (b) \mu _{\bs (a)} (b)
= \int _{c \in \bs ^{-1} (\bt (a))} \psi _1 (c ^{-1}) \psi _2 (c a) \mu _{\bt (a)} (c),$$
one can without loss of generality assume $\varepsilon _1 \leq \varepsilon _2$.
Then by definition one has the estimates 
$ \psi _1 (a) \leq M e ^{- \varepsilon ' _1 d (a, \bs (a))} ,
\psi _2 (a) \leq M' e ^{- \varepsilon ' _2 d (a, \bs (a))} $
for some $\varepsilon '_1 > \varepsilon _1 , \varepsilon ' _2 > \varepsilon _2 $ and constants $M , M' > 0$.
One may further assume that $\varepsilon ' _1 < \varepsilon ' _2 $.

The hypothesis implies for any $a \in \cG$
\begin{align*}
| \psi _1 \circ \psi _2 (a) |
\leq & M M'\int _{b \in \cG _{\bs (a)}} 
e ^{- \varepsilon ' _1 d (a , b) } e ^{ - \varepsilon ' _2 d (b , \bs (b))} \mu _{\bs (a)} (b) \\
\leq & M M' \int _{b \in \cG _{\bs (a)}} 
e ^{- \varepsilon ' _1  |d (a , \bs (a)) - d (b , \bs (b))| - \varepsilon ' _2 d (b , \bs (b))} \mu _{\bs (a)} (b) \\
= & M M' \int _{b \in B _a } 
e ^{- \varepsilon ' _1  d (a , \bs (a)) } e ^{- (\varepsilon ' _2 - \varepsilon ' _1) d (b , \bs (b))} 
\mu _{\bs (a)} (b) \\ 
&+ M M' \int _{b \not \in B _a} 
e ^{ \varepsilon ' _1 d (a , \bs (a)) } e ^{- (\varepsilon ' _2 + \varepsilon ' _1) d (b , \bs (b))} \mu _{\bs (a)} (b),
\end{align*}
where $B _a$ denotes the set $\{ b \in \cG _{\bs (a)} : d (b , \bs (b)) < d (a , \bs (a)) \} $
for each $a$.
Hence for the first integral, one has
\begin{align*}
\int _{b \in B _a } 
e ^{- \varepsilon ' _1  d (a , \bs (a)) } & e ^{- (\varepsilon ' _2 - \varepsilon ' _1) d (b , \bs (b))} 
\mu _{\bs (a)} (b) \\
\leq & \: e ^{- \varepsilon ' _1  d (a , \bs (a)) } 
\int _{b \in \cG _{\bs (a)}} e ^{- (\varepsilon ' _2 - \varepsilon ' _1) d (b , \bs (b))} \mu _{\bs (a)} (b),
\end{align*}
which is finite and only depends on $\bs (a)$ (since $\cG$ is only of polynomial growth).
As for the second integral, write
\begin{align*}
\varepsilon ' _1  d (a , \bs (a)) -  &(\varepsilon ' _2 + \varepsilon ' _1) d (b , \bs (b)) \\
= &- \varepsilon ' _1  d (a , \bs (a)) 
+ 2 \varepsilon ' _1 ( d (a , \bs (a)) - d (b, \bs (b)))
- (\varepsilon ' _2 - \varepsilon ' _1) d (b , \bs (b)).
\end{align*}
Since $d (b, \bs (b)) \geq d (a , \bs (a)) $ for any $b \not \in B _a$.
It follows that the second integral is again bounded by
$$ e ^{- \varepsilon ' _1 d (a , \bs (a))}
\int _{b \in \cG _{\bs (a)}} e ^{- (\varepsilon ' _2 - \varepsilon ' _1) d (b , \bs (b))} \mu _{\bs (a)} (b).$$
Adding the two together and rearranging, one gets
$ e ^{\varepsilon ' _1 d _\bs (a)} (u _1 \circ u _2 ) (a) $ is a bounded function, as asserted.

To prove the assertion for derivatives,
observe that by right invariance of $\mu$,
$$ (\psi _1 \circ \psi _2 )(a b^{-1}) = \int \psi _1 (a c ^{-1} ) \psi _2 (c b ^{-1}) \mu _{\bs (a)} (c) ,$$
for any $(a, b) \in \tilde \cG $.
It follows that for any $(a , b ) \in \tilde \cG $, $X , Y \in \Gamma ^\infty (\rA)$,
$$ L _{d \tilde \bm (X ^\br \oplus Y ^\br )} (\psi _1 \circ \psi _2 )(a b ^{-1}) 
= \int ( L _{d \tilde \bm (X ^\br \oplus 0 )} \psi _1 )(a c ^{-1} ) 
( L _{d \tilde \bm (0 \oplus Y ^\br )} \psi _2 )(c b ^{-1}) \mu _{\bs (a)} (c). $$
and so on for higher derivatives.
\end{proof}

Next, we write down the definition of the calculus with bounds.
For each $k$, denote 
$$ \hat \rho _k := ( (\bs ^* \rho _k )^2 + (\bt ^* \rho _k )^2 )^{\frac{1}{2}} \in C ^0 (\cG).$$
\begin{dfn}
\label{FFDegen}
Let $\cG$ be uniformly degenerate. For each 
$\varepsilon > 0, \lambda _1 , \cdots , \lambda _r \geq 0 $,
the calculus with bounds of order  $- \infty $ is defined to be 
\begin{align*}
\Psi ^{- \infty } _{\varepsilon _1 ; \lambda _1 , \cdots , \lambda _r } (\cG) 
:= \Big\{ \psi \in C ^0 ( \cG ) : \psi |_{\cG _k } \in C ^\infty &( \cG _k),
\Forall l \in \bbN, \exists M _l > 0 \text { such that } \\
\text{ for any } X _1 , \cdots , X _l &, Y _1 , \cdots Y _l \in \Gamma ^\infty (\rA ), (a , b) \in \tilde \cG \\
\tilde \nabla ^l (\tilde \bm ^* \psi ) (X _l ^\br & \oplus Y _l ^\br , \cdots , X _1 ^\br \oplus Y ^\br _1 )(a , b ) \\
\leq M _l & e ^{- \varepsilon ' d (a , b) }
\prod _{i = 1 }^l (|X_i | + |Y_i|) \prod_{i=1} ^l \hat \rho _i ^{\lambda _i } \Big\}.
\end{align*}
\end{dfn}

\begin{rem}
Loosely speaking, 
$e ^{- d (a , \bs (a)) } $ in our calculus with bounds plays the role of left and right boundary in \cite{Melrose;Book};
while $\hat \rho _k $ plays the role of the front face defining function.
\end{rem}

With the new filtration we need to refine the composition rule.
\begin{thm} 
\label{CompoThm}
Given any collection of data $\varepsilon _1, \varepsilon _2 > 0, 
\lambda ^{(1)} _1 , \cdots , \lambda ^{(1)} _r, \lambda ^{(2)} _1 , \cdots, \lambda ^{(2)} _r \geq 0 $.
Suppose that 
$$\varepsilon _3 := \min \Big\{ \varepsilon _1 - \sum _{k = 1} ^r \omega _k \lambda ^{(2)} _k ,
\varepsilon _2 - \sum _{k =1} ^r \omega _k \lambda ^{(1)} _k \Big\} > 0 ,$$
with $\omega _ k $ is as in Theorem \ref{DEst}.
Then the convolution product between any pair of elements in 
$\Psi ^{-\infty } _{\varepsilon _1 ; \lambda ^{(1)} _1 , \cdots , \lambda ^{(1)} _r } (\cG) $ and 
$\Psi ^{-\infty } _{\varepsilon _2 ; \lambda ^{(2)} _1 , \cdots , \lambda ^{(2)} _r } (\cG) $
is well defined.
Moreover, one has
\begin{equation}
\Psi ^{-\infty } _{\varepsilon _1 ; \lambda ^{(1)} _1 , \cdots , \lambda ^{(1)} _r } (\cG) 
\circ \Psi ^{-\infty } _{\varepsilon _2 ; \lambda ^{(2)} _1 , \cdots , \lambda ^{(2)} _r } (\cG)
\subseteq \Psi ^{-\infty } _{\varepsilon _3; 
\lambda ^{(1)} _1 + \lambda ^{(2)} _1, \cdots , \lambda ^{(1)} _r + \lambda ^{(2)} _r } (\cG). 
\end{equation}
\end{thm}
\begin{proof} 
For any $\psi _1 , \in \Psi ^{-\infty } _{\varepsilon _1 ; \lambda ^{(1)} _1 , \cdots , \lambda ^{(1)} _r } (\cG) $,
$\psi _2 , \in \Psi ^{-\infty } _{\varepsilon _2 ; \lambda ^{(2)} _1 , \cdots , \lambda ^{(2)} _r } (\cG) $,
the convolution product reads (provided the integral is finite):
\begin{align*} 
\psi _1 \circ \psi _2 (a) &=  \int _{b \in \cG _{\bs (a)}} u _1 (a b ^{-1}) u _2 (b) \mu _{\bs (a)} (b) \\
\leq & M _1 \int _{b \in \cG _{\bs (a)}} 
e ^{- \varepsilon ' _1 d (a , b) } e ^{ - \varepsilon ' _2 d (b , \bs (b))} \\
\times & \prod _{i =1 } ^r 
\big( (\bs ^* \rho _i (a b ^{-1}))^2 + (\bt ^* \rho _i (a b ^{-1}))^2 \big) ^{\frac{\lambda ^{(1)} _i}{2}}
\big( (\bs ^* \rho _i (b ))^2 + (\bt ^* \rho _i (b) )^2 \big) ^{\frac{\lambda ^{(2)} _i}{2}} \mu _{\bs (a)} (b).
\end{align*}
For some $\varepsilon ' _1 > \varepsilon _1 , \varepsilon '_2 > \varepsilon _2$.
Consider the product term in the integrand, one has 
\begin{align*} 
\bs ^* \rho _i (a b ^{-1} ) &= \rho _i (\bt (b)) \leq e ^{\omega _i d (b, \bs (b) )} \rho _i (\bs (a)) \\
\bt ^* \rho _i (a b ^{-1} ) &= \rho _i (\bt (a)) \\
\bt ^* \rho _i (b ) &= \rho _i (\bt (b)) \leq e ^{\omega _i d (b, a )} \rho _i (\bt (a)) \\
\bs ^* \rho _i (b ) &= \rho _i (\bs (a)),
\end{align*}
where we used Theorem \ref{DEst} for the first and third line.
Hence, one estimates the integrand
\begin{align*}
e ^{- \varepsilon ' _1 d (a , b) } & e ^{ - \varepsilon ' _2 d (b , \bs (b))} \prod _{i=1 } ^r 
\big( (\bs ^* \rho _i (a b ^{-1}))^2 + (\bt ^* \rho _i (a b ^{-1}) )^2 \big) ^{\frac{\lambda ^{(1)} _i}{2}}
\big( (\bs ^* \rho _i (b ))^2 + (\bt ^* \rho _i (b))^2 \big) ^{\frac{\lambda ^{(2)} _i}{2}} \\
& \leq \Big( \prod _{i=1 } ^r \hat \rho _i (a) ^{\lambda ^{(1)} _i + \lambda ^{(2)} _i} \Big)
e ^{- (\varepsilon ' _1 - \sum _{i=1} ^r \omega _i \lambda ^{(2)} _i )d (a , b) }
e ^{- (\varepsilon ' _2 \sum _{i=1} ^r \omega _i \lambda ^{(1)} _i ) d (b , \bs (b))} .
\end{align*}
The theorem follows by factoring out the term 
$\prod _{i=1 } ^r \hat \rho _k (a) ^{\lambda ^{(1)} _k + \lambda ^{(2)} _k} $,
and then following the arguments of Lemma \ref{CompoLem}.
\end{proof}

In the case $\cG$ is uniformly degenerate, since $\omega _k $ can be made arbitrary small,
it follows that convolution is always defined and
\begin{equation}
\Psi ^{-\infty } _{\varepsilon _1 ; \lambda ^{(1)} _1 , \cdots , \lambda ^{(1)} _r } (\cG) 
\circ \Psi ^{-\infty } _{\varepsilon _2 ; \lambda ^{(2)} _1 , \cdots , \lambda ^{(2)} _r } (\cG)
\subseteq \Psi ^{-\infty } _{\min \{ \varepsilon _1 , \varepsilon _2 \}; 
\lambda ^{(1)} _1 + \lambda ^{(2)} _1, \cdots , \lambda ^{(1)} _r + \lambda ^{(2)} _r } (\cG). 
\end{equation}

We turn to study convolution of singular kernels.
\begin{lem}
\label{DiffIdeal}
For any $m \in \bbZ , \varepsilon > 0, \lambda _1 , \cdots , \lambda _r \geq 0 $,
\begin{equation}
\Psi ^{[m]} (\cG ) \circ \Psi ^{- \infty } _{\varepsilon _1 ; \lambda _1 , \cdots , \lambda _r } (\cG) 
\subset \Psi ^{-\infty } _{\varepsilon _1 ; \lambda _1 , \cdots , \lambda _r } (\cG) .
\end{equation}
\end{lem}
\begin{proof}
Suppose we are given kernels $\psi \in \Psi ^{[m]} (\cG), 
\kappa \in \Psi ^{- \infty } _{\varepsilon _1 ; \lambda _1 , \cdots , \lambda _r } (\cG) $.
Fix any elliptic differential operator $\Delta \in \Psi ^{[m']} (\cG )$ with order $m' > m + \dim \cG - \dim \rM $.
Let $Q \in \Psi ^{[- m']} (\cG )$ be a parametrix of $\Delta $. In other words, one has
$$ S := \id - Q \Delta \in \Psi ^{- \infty} (\cG ).$$  
Then one can write
$$ \psi \circ \kappa
= ( \psi S ) \circ \kappa + ( \psi Q ) \circ (\Delta \kappa ).$$
By definition, \ref{FFDegen},
$\kappa \in \Psi ^{- \infty } _{\varepsilon _1 ; \lambda _1 , \cdots , \lambda _r } (\cG) $
implies $\Delta \kappa $ lies in the same space.

On the other hand, $\psi Q$ is a uniformly supported pseudo-differential operator of order less that 
$( - \dim \cG + \dim \rM )$.
Therefore it is a classical result (see \cite{Hormander;3}) that 
$ \psi Q $ is continuous kernel on $\cG$ with compact support.
It follows that the proof of \ref{CompoThm} applies and 
$( \psi Q ) \circ (\Delta \kappa ) \in \Psi ^{- \infty } _{\varepsilon ; \infty } (\cG) $. 
The same argument holds for $( \psi S ) \circ \kappa $.
Hence we conclude that 
$$ \psi \circ \kappa \in \Psi ^{- \infty } _{\varepsilon _1 ; \lambda _1 , \cdots , \lambda _r } (\cG) ,$$
as asserted.
\end{proof}

By considering adjoint of Lemma \ref{DiffIdeal}, it is obvious that
\begin{equation}
\Psi ^{- \infty } _{\varepsilon _1 ; \lambda _1 , \cdots , \lambda _r } (\cG) \circ \Psi ^{[m]}
\subset \Psi ^{-\infty } _{\varepsilon _1 ; \lambda _1 , \cdots , \lambda _r } (\cG) ,
\end{equation}
as well.

\begin{nota}
For $1 \leq k \leq r - 1 $, denote 
\begin{equation}
\Psi ^{-\infty } _{\varepsilon ; \lambda _1 , \cdots \lambda _k , \infty} (\cG)
:= \bigcap _{ \lambda _{k + 1} > 0 } 
\Psi ^{-\infty } _{\varepsilon ; \lambda _1 , \cdots \lambda _k , \lambda _{k+1} , \cdots , \lambda _r } (\cG).
\end{equation}
Note that $\Psi ^{-\infty } _{\varepsilon ; \lambda _1 , \cdots \lambda _k , \infty } (\cG)$   
is uniquely defined since $\rho _1 \leq \rho _2 \leq \cdots \leq \rho _r $.

If for some $\lambda _1 = \cdots = \lambda _{r' } = 0$ for some $r ' \leq r $, 
we shall denote $\Psi ^{- \infty } _{\varepsilon ; \lambda _1 , \cdots \lambda _r } (\cG)$ by
$$ \Psi ^{-\infty } _{\varepsilon ; \bzero _{r' } , \lambda _{r' + 1} , \cdots \lambda _r } (\cG).$$
For convenience, we shall also write
$$ \Psi ^{- \infty } _{\bullet , \bzero _{r'} , \lambda _{r' + 1}, \cdots , \lambda _r } (\cG ) 
:= \bigcup _{\varepsilon > 0 } 
\Psi ^{- \infty } _{\varepsilon ; \bzero _{r'} , \lambda _{r' + 1}, \cdots , \lambda _r } (\cG ). $$
\end{nota}

\begin{lem}
For any $\psi \in \Psi ^{-\infty } _{\varepsilon ; \lambda _1 , \cdots \lambda _k , \infty} (\cG)$,
all derivatives of $\psi $ vanish at $\bar \cG _{k + 1} $ and $\psi |_{\cG _k }$ is smooth.
\end{lem}
\begin{proof}
It suffices to consider the derivatives on some coordinates patches.
For any $a \in \bar \cG _{k+1} $, let 
$$ \bx ^{(\alpha )} _{Z _I } (\tau , x)
:= \exp ( \tau \cdot ( X ^{(\alpha )} _1 , \cdots , X ^{(\alpha )} _k ) ) \exp Z _I (x) $$
be an exponential coordinates patch defined in Definition \ref{ExpCoord} around $a$.
We may assume that there exist constant $C > 0 $ such that $d (b , \bs (b)) < C $ on $U ^{(\alpha )} _{Z _I }$.
Then 
$$\hat \rho ^2 _{k+1} (\bx ^{(\alpha )} _{Z _I } (\tau , x)) \leq 
\rho ^2 _{k+1} (\bx ^{(\alpha )} _{Z _I } (\tau , x)) 
+ e ^{C \omega } \rho ^2 _{k+1} (\bx ^{(\alpha )} _{Z _I } (\tau , x)) ,$$
for some $\omega > 0 $.
Therefore the assumption 
$\psi \in \Psi ^{-\infty } _{\varepsilon ; \lambda _1 , \cdots \lambda _k , \infty} (\cG)$
implies 
$$ \psi \leq (1 + e ^{C \omega }) (x ^2 _{\dim \cG _{k+1} +1} + \cdots + x ^2 _{\dim \cG })^N $$ 
for any $N > 0 $, which in turn implies all derivatives of $\psi $ at the subset
$\{ x _1 = \cdots = x _{\dim \cG _{k+1} } \} $ vanishes.

Since by definition, $\psi $ is smooth on $\cG _k$ and $\bar \cG _k = \cG _k \bigcup \bar \cG _{k+1} $,
it follows that $\psi $ is smooth on $\bar \cG _k $. 
\end{proof}

\begin{dfn}
Given $r $ sequences of number $\{ \lambda ^{(1)} _i \} , \cdots , \{ \lambda ^{(r)} _i \} $, 
such that for all $i $, $\lambda ^{(k)} _1 = 0$ , and $ \{ \lambda ^{(k)} _i \} \to \infty $ as $i \to \infty $.
Given $r$ sequences of kernels $\{ \psi ^{k} _i \} , k = 1 , \cdots , r$, such that
$$\psi ^{(k)} _i \in \Psi ^{- \infty } _{\varepsilon ; \bzero _{k - 1} , \lambda ^{(k)} _i , \infty } (\cG ),$$
and furthermore satisfy the smoothness conditions
$$ \psi ^{(k)} _1 \in C^\infty (\cG ) ,
\quad \psi ^{(k)} _i |_{\cG \setminus \bar \cG _k } \in C ^\infty (\cG \setminus \bar \cG _k ), 
\text { for all } i \geq 2.$$
We say that an element $\psi \in \Psi ^{- \infty } _{\varepsilon ; \bzero }(\cG ) $ has an asymptotic expansion
$$ \psi \sim \sum _{k = 1} ^r \Big(\sum _{i = 1} ^\infty \psi _i ^{(k)} \Big) ,$$
if for any sets of indexes $i _1 , \cdots , i _r $, 
$$ \psi - \sum _{k = 1} ^r \Big(\sum _{i = 1} ^{i _k} \psi _i ^{(k)} \Big) 
\in \Psi ^{- \infty } _{\varepsilon ; \lambda ^{(1)} _{1 + i _1 } , \cdots , \lambda ^{(r)} _{1 + i _r}} (\cG).$$
The space of kernels with asymptotic expansion is denoted by 
$ \Psi ^{- \infty } _\varepsilon (\cG ) $, and we write 
$\Psi ^{- \infty } _\bullet (\cG ) := \bigcap _{\varepsilon > 0 } \Psi ^{- \infty } _\varepsilon (\cG )$.
\end{dfn}

\section{Compact parametrix and generalized inverse of Fredholm operators}

\subsection{Extension of exponentially decaying kernels}
The following assumption is crucial in our construction of parametrices and inverses of uniformly supported 
pseudo-differential operator on $\cG$.
\begin{dfn}
\label{ExtDfn}
Let $\cG$ be a boundary groupoid, not necessary uniformly degenerate.
We say that $\cG$ satisfies the {\it extension property} if
for any $k \leq r , \varphi \in \Psi ^{- \infty} _{\bullet ; \infty } (\bar \cG _k ), 
\varphi _0 \in \Psi ^{- \infty } _{\bullet ; \bzero _k } (\cG ), \kappa \in \Psi ^{- \infty } (\cG ) $, 
and differential operator $D \in \Psi ^{[m]} (\cG )$ satisfying the relation
$$ D |_{\bar \cG _k} ( \varphi _0 | _{\bar \cG _k} + \varphi ) - \kappa |_{ \bar \cG _k} = 0 ,$$
there exists an extension $\bar \psi \in \Psi ^{- \infty} _{\bullet ; \bzero _k , \infty } (\cG )$ such that
\begin{enumerate}
\item $\bar \varphi |_{\bar \cG _k } = \varphi ;$
\item $ D (\varphi _0 + \bar \varphi ) - \kappa 
\in \Psi ^{- \infty } _{\bullet ; \bzero _{k - 1} , \lambda _k , \infty } (\cG ),$
for some $\lambda _k > 0 $.
\end{enumerate}

We say that $\cG $ has the smooth extension property if for any
$\varphi \in \Psi ^{- \infty} _{\bullet ; \infty } (\bar \cG _k ), \varphi _0$
with asymptotic expansion $\varphi _0 \sim \sum _{j = k + 1 } ^r \sum _{i = 1} \varphi ^{(j)} _i $,
$\varphi ^{(j)} _i \in \Psi ^{- \infty } _{\varepsilon ; \bzero _j , \lambda ^{(j )} _i , \infty } (\cG )$,
and pseudo-differential operator $\varPsi \in \Psi ^{[m]} (\cG )$ satisfying the relation
$$ \varPsi |_{\bar \cG _k} ( \varphi _0 | _{\bar \cG _k} + \varphi ) - \kappa |_{\bar \cG _k} = 0 ,$$
there exists an extension 
$\bar \varphi \in \Psi ^{- \infty} _{\bullet ; \bzero _k , \infty } (\cG ) \bigcap C ^\infty (\cG )$ such that
$\kappa - \varPsi ( \varphi _0 + \bar \varphi ) 
\in \Psi ^{- \infty } _{\bullet ; \bzero _{k- 1} , \lambda , \infty } (\cG )$
for some $\lambda > 0$.
\end{dfn}
In the appendix (See Propositions \ref{Proof1} and \ref{Proof2}), we shall prove that
\begin{thm}
\label{SimpleExt}
Any boundary groupoid of the form
$$ \cG = ( \rM _0 \times \rM _0 ) \bigsqcup \rG \times (\rM _1 \times \rM _1 ), $$
where $\rG $ is a nilpotent Lie group, satisfies the extension property.
\end{thm}
The proof is elementary (but tedious). 
Indeed we conjecture that:
\begin{conj}
\label{MainConj}
Every boundary groupoid $\cG$ with polynomial volume growth satisfies the smooth extension property.
\end{conj}
In the following we shall assume that the groupoid $\cG$ satisfies the extension property.

\subsection{Inverse and parametrix when $\cG = \rM _0 \times \rM _0 \bigsqcup \rM _1 \times \rM _1 \times \rG $}
The main tool we use is the following estimate from Shubin:
\begin{lem}
\label{ZeroDecay}
{\em \cite{Shubin;BdGeom}}
Let $\rM $ be any manifold with bounded geometry. Fix $r > 0 $.
For any invertible, uniformly supported pseudo-differential operators $\varPsi $ of order $m$ on $\rM$,
$\varPsi ^{-1} $ is pseudo-differential operator of order $- m$,
where the distributional kernel $\varphi $ of $\varPsi ^{-1} $ satisfies the estimate 
\begin{equation}
\partial ^I _x \partial ^J _y \varphi (x, y) 
\leq M _{I J} (1 + d (x, y)^{m - |I| - |J| - \dim \rM }) e ^{- \varepsilon d (x , y)}, 
\end{equation}
for some $\varepsilon > 0 $ and for all multi-indices $I, J$ and all $(x, y) \in \rM \times \rM \setminus \rM $.
Moreover, if $m - \dim \rM > 0 $, then $\varphi \in C ^j (\rM \times \rM )$ for any $j < m - \dim \rM $.
\end{lem}

The first step of our construction is to describe a compact parametrix of an elliptic operator.
\begin{thm}
\label{CptLem1}
Let $ \cG = \rM _0 \times \rM _0 \bigsqcup \rG \times \rM _1 \times \rM _1 ,$
Given any elliptic, groupoid differential operator 
$\varPsi = \{ \varPsi \} _{x \in \rM} \in \Psi ^{[m]} (\cG )$.
Suppose that for all $x \not \in \rM _0$, $\varPsi _x $ is invertible,
then its vector representation $\nu (\varPsi)$ is Fredholm.
Moreover, there exists $\varPhi \in \Psi ^{-[m]} (\cG) \oplus \Psi ^{- \infty} _{\bullet ; 0} (\cG) $ such that
$$ \id - \nu ( \varPsi ) \nu ( \varPhi ) : \mathcal L ^2 (\rM _0 ) \to \mathcal L ^2 (\rM _0 )$$ 
is compact.
\end{thm}
\begin{proof}
The proof of the theorem closely follows the proof of \cite[Theorem 3.21]{So;PhD}.

By standard arguments there exists $Q \in \Psi ^{[- m]} (\cG) $  such that
$$ \id - \varPsi Q , \id - Q \varPsi \in \varPsi ^{- \infty } (\cG ). $$
Observe that $( \varPsi |_{\cG _1 } )^{-1}$ is right invariant by uniqueness of the inverse operator.
By Lemma \ref{ZeroDecay}, one has 
$$( \varPsi |_{\cG _1 } )^{-1} \in \Psi ^{[- m]} (\cG _1 ) \oplus \Psi ^{- \infty } _\varepsilon (\cG _1), $$
for some $\varepsilon > 0 $.
It follows that 
$$ ( \varPsi |_{\cG _1 } )^{-1} - Q |_{\cG _1 } \in \Psi ^{- \infty } _\varepsilon (\cG _1). $$
The extension property guarantees that there exists $S \in \Psi ^{- \infty } _{\varepsilon ' ; 0} (\cG )$,
for some $\varepsilon ' > 0$, such that 
$$ S |_{\cG _1 } =  ( \varPsi |_{\cG _1 } )^{-1} - Q |_{\bar \cG _1 } .$$
Define
\begin{equation}
\label{SimplePara}
\varPhi := Q + S .
\end{equation}
Then $ \varPhi |_{\cG _1 } = ( \varPsi |_{\cG _1 } )^{-1} $. It follows that
$$ (\id - \varPsi \varPhi )|_{\cG _1} = \id - (\varPsi |_{\cG _1} ) (\varPhi |_{\cG _1} ) = 0 
= (\id - \varPhi \varPsi)|_{\cG _1}.$$
Therefore by \cite{Nistor;GeomOp}, $\id - \nu ( \varPsi ) \nu ( \varPhi )$ is compact.
\end{proof}

Assume further that $\cG$ is uniformly degenerate. 
Then one can improve the parametrix by considering the Neumann series, as in \cite{Mazzeo;EdgeRev}.

\begin{lem}
\label{InftyParaLem}
Let $\varPsi \in \Psi ^{[m]} (\cG )$ be as in Theorem \ref{CptLem1}.
Then there exists $ \varPhi ' \in \Psi ^{- \infty } _{\bullet ; 0} (\cG )$ such that
$$ \id - \varPsi (\varPhi + \varPhi ' ) \in \Psi ^{- \infty } _{\bullet ; \infty } (\cG ).$$
\end{lem}
\begin{proof} 
Let $\varPhi $ be defined in Equation (\ref{SimplePara}), $\varphi$ be the reduced kernel of $\varPhi$.
Then Theorem \ref{CptLem1} implies that regarded as a kernel,
$$ R := \id - \psi \circ \varphi \in \Psi ^{- \infty } _{\varepsilon ; \lambda } (\cG ) ,$$
for some $\varepsilon > 0, \lambda > 0$. 
Using the arguments in the proof of Theorem \ref{CompoThm} repeatedly, one has for any $k \in \bbN$,
$$ \nabla ^l \varphi \circ R ^k (a) 
\leq M' _l M ^{k \lambda } e ^{ - \varepsilon d (a , \bs (a))} \hat \rho ^{k \lambda } ,$$
for some constants $M, M' _l $.
In particular, on an open neighborhood of $\cG _1 $ where $\hat \rho $ is sufficiently small,
$$ \sum _{k =1 } ^N k \tilde \nabla ^l (\tilde \bm ^* (\varphi \circ R ^k ))$$ 
converges uniformly and absolutely for all $l$.
Define $ \varPhi ' $ to be the limit 
\begin{equation}
\label{InftyPara}
\varPhi ' (a) := \sum _{k = 1} ^\infty \theta (k \lambda \hat \rho (a) ) (\varphi \circ R ^{ k } )(a),
\end{equation}
where $\theta \in C ^\infty (\bbR )$ is a function equal to $1$ on a neighborhood of $0$
with sufficiently small support.
Observe that $\varPhi ' \in \Psi ^{- \infty } _{\varepsilon ; 0} (\cG) $ 
since $\sum _{k} \theta (\hat \rho ) M' _l M ^{k \lambda } \hat \rho ^{k \lambda }$ converges absolutely and uniformly.
Moreover, for all $N = 1, 2, \cdots $,
$$ \varPhi ' - \sum _{k =1} ^N \varPhi R ^{k } \in \Psi ^{- \infty } _{\varepsilon ; N \lambda } (\cG).$$
Since
$\id - \varPsi \varPhi (\id + R + R ^2 + \cdots + R ^{N -1}) = R ^N$,
which lies in $\Psi ^{- \infty } _{\varepsilon ; N \lambda } (\cG)$ by Theorem \ref{CompoThm},
it follows that 
$$ \id - \varPsi (\varPhi + \varPhi ' ) \in \bigcap _{N \in \bbN } \Psi ^{- \infty } _{\varepsilon ; N \lambda } (\cG)
= \Psi ^{- \infty } _{\varepsilon ; \infty } (\cG).$$
\end{proof}

\begin{rem}
\label{InftyParaRem}
By similar arguments, 
one gets $\tilde \varPhi \in \Psi ^{[- m]} (\cG ) \oplus \Psi ^{- \infty } _\varepsilon (\cG ) $ such that
$$ \id - \tilde \varPhi \varPsi \in \Psi ^{- \infty } _{\varepsilon ; \infty } (\cG ).$$
\end{rem}

\begin{rem}
In the case of $\cG $ being uniformly non-degenerate, 
one may follow the arguments as in \cite{Mazzeo;EdgeRev}. 
Note in that case the author needs an extra step to modify the parametrix $\varPhi $ so that 
$( \id - \varPsi \varPhi ) ^N $ is well defined for all $N$.  
\end{rem}

\begin{thm}
\label{InvDecay}
Let $ \cG = \rM _0 \times \rM _0 \bigsqcup \rG \times \rM _1 \times \rM _1 ,$ 
be uniformly degenerate and with polynomial growth.
For any uniformly supported, elliptic operator $\varPsi \in \Psi ^{[m]} (\cG )$,
suppose that $\varPsi $ is invertible. 
Then $\varPsi ^{-1} \in \Psi ^{- [m]} _{\varepsilon } (\cG ) $ for some $\varepsilon > 0$.
\end{thm}
\begin{proof}
Regard $\{ \varPsi _x \} _{ x \in \rM _0 }$ as a (pseudo-)differential operator on the manifold with bounded geometry 
$\rM _0 $, or in other words, a  kernel on $\rM _0 \times \rM _0 $.
Then Lemma \ref{ZeroDecay} again applies: $\varPsi _x ^{-1} $ is a uniform pseudo-differential operator on $\rM _0 $.
Let $\psi _0 $ be the kernel of $\varPsi _x , x \in \rM _0$, $\phi $ be the reduced kernel of $\varPsi _x ^{-1} $.

Let $\bar \varPhi := \varPhi + \varPhi '$ be defined in the previous lemma (Equation (\ref{InftyPara})) with reduced 
kernel $\bar \varphi $. 
Consider 
$$ \bar \varphi |_{\cG _0 } = \phi \circ ( \psi |_{\cG _0 } ) \circ ( \bar \varphi |_{\cG _0 })
= \phi - \phi \circ (\kappa |_{\cG _0} ),$$
where $\kappa $ is the reduced kernel of $ \id - \varPsi (\varPhi + \varPhi ' )$.

We use similar arguments as in Lemma \ref{DiffIdeal}.
Let $\Delta \in \Psi ^{[m']} (\cG )$ be any elliptic differential operator with order $m' > m + \dim \rM _0 $.
Let $Q \in \Psi ^{[- m']} (\cG )$ be a parametrix of $\Delta $, $S := \id - Q \Delta \in \Psi ^{[m']} (\cG )$.  
Then one can write
$$ \phi \circ (\kappa |_{\cG _0} ) 
= \phi ( S |_{\cG _0 } ) \circ \kappa + \phi (Q | _{\cG _0} ) \circ (\Delta \kappa ) |_{\cG _0 }.$$

Since $\kappa \in \Psi ^{- \infty } _{\varepsilon ; \infty } (\cG),
\Delta \kappa \in \Psi ^{- \infty } _{\varepsilon ; \infty } (\cG).$
On the other hand, $\phi (Q | _{\cG _0} )$ is a uniform pseudo-differential operator of order less that 
$ - \dim \rM _0 $.
Therefore by Lemma \ref{ZeroDecay} $ \phi (Q | _{\cG _0} )$ is bounded, continuous on $\rM _0 \times \rM _0$ 
and decays exponentially.
It follows that the proof of \ref{CompoThm} applies and 
$\phi (Q | _{\cG _0} ) \circ (\Delta \kappa ) |_{\cG _0 } \in \Psi ^{- \infty } _{\varepsilon ; \infty } (\cG) $ 
(extending the kernel to $\cG $ by $0$).
The same argument holds for $\phi ( S |_{\cG _0 } ) \circ \kappa $.
Hence we conclude that 
\begin{equation}
\label{InvRecursion}
\phi = \bar \varphi + \phi ( S |_{\cG _0 } ) \circ \kappa + \phi (Q | _{\cG _0} ) \circ (\Delta \kappa ) |_{\cG _0 }
\in \Psi ^{[- m]} (\cG ) \oplus \Psi ^{- \infty } _\varepsilon (\cG ).
\end{equation}
\end{proof}

\subsection{The generalized inverse}
Given a Fredholm operator $T$ on a Hilbert space, 
it is a standard fact that both the null space and co-kernel of $T$ are finite dimensional.
Moreover $T $ is invertible modulo projection onto its null space and co-kernel.
In this section, let $\cG$ be uniformly degenerate, $\varPsi \in \Psi ^{[m]} (\cG), m \geq 0 $ elliptic,
so that $\nu (\varPsi ) : \mathcal W ^m (\rM _0 ) \to \mathcal L ^2 (\rM _0 ) $ is Fredholm.
We describe the null space of $\nu (\varPsi )$.

Let $ G : \mathcal L ^2 (\rM _0 ) \to \mathcal W ^m (\rM _0 )$ be the generalized inverse of $\nu (\varPsi ) $.
In other words,
\begin{equation}
\id - G \nu (\varPsi ) = P _0 , \quad \id - \nu (\varPsi ) G = P _\bot ,
\end{equation}
where $P _\bot $ and $P _0 $ are the projection operators onto the null space and co-kernel of 
$\nu (\varPsi )$ respectively.

To describe $P _0 $ and $P _\bot$, we give an a-prior estimate of the null space of $\nu (\varPsi )$ 
(and that of co-kernel of $\nu (\varPsi )$) is similar.
\begin{lem}
Given any $f \in \cL ^2 (\rM _0 ) $ such that $\nu (\varPsi )( f ) = 0$.
Then for any $f |_{\rM _0 }$ is smooth and for any $X _1 , \cdots X _l \in \Gamma ^\infty (\rA ) , N \in \bbN,$
$\rho ^{- N} L _{\nu (X _1)} \circ \cdots \circ L_{\nu (X _l)} f $ is bounded.
\end{lem}

\begin{proof}
Let $\tilde \varPhi ' $ be a parametrix of $\varPsi $ such that 
$R := \id - \tilde \varPhi \Psi \in \Psi ^{- \infty } _{\bullet ; \infty } (\cG )$,
as in Lemma \ref{InftyPara}.
Then $\nu (R ) f = f - \nu (\tilde \varPhi ) \nu (\varPsi ) f$.
Since $\nu (R ) f $ is smooth on $\rM _0 $, $f $ is smooth on $\rM _0 $.

Moreover, for any $X _1 , \cdots X _l \in \Gamma ^\infty (\rA ) , N \in \bbN,$
let $S := L _{X _1 } \circ \cdots \circ L_{ X _l } R $, 
then $S \in \Psi ^{- \infty } _{\varepsilon ; \infty } (\cG )$ for some $\varepsilon > 0$.
Therefore for any $N$, $a \in \cG$, 
\begin{align*}
|(\bt ^* \rho (a)) ^{- N} & S (\bt ^* f ) (a)| \\
\leq & (\bt ^* \rho (a)) ^{- N} M 
\int e ^{- \varepsilon d (a , b)} ( (\rho (\bt (a))) ^2 + (\rho (\bt (b)))^2 )^\frac{N}{2} (\bt ^* f ) (b) 
\mu _{\bs (a)} (b) \\
\leq & (\bt ^* \rho (a)) ^{- N} M' 
\int e ^{- \varepsilon d (a , b)} (1 + e ^{ \frac{\varepsilon d (a , b)}{N}})^{\frac {N}{2}}
(\bt ^* \rho (a)) ^N (\bt ^* f ) (b) \mu _{\bs (a)}(b),
\end{align*}
for some constants $M , M'$. 
Since by definition, 
we have $\cG _0 = \rM _0 \times \rM _0 $ and the $\bs$-fiber is equipped with the same Riemannian density as $\rM _0 $,
it follows from the polynomial growth of $\cG$ and Cauchy-Schwartz inequality that the integral above is bounded 
independent of $a$. 
Hence the claim.
\end{proof}

Define $ \mathcal S (\cG _0) \subset C ^\infty ( \cG _0 ) $ 
to be the space of Schwartz functions on $ \cG _0 $ with respect to $\rho $.
In other words $ \phi \in \mathcal S (\cG _0 ) $ if and only if for all $l, N, N' \in \bbN $,
\begin{align*}
\nabla ^l \phi &
(d \tilde \bm (X _1 ^\br \oplus Y _1 ^\br ) , \cdots , d \tilde \bm (X _l ^\br \oplus Y _l ^\br )) (a b ^{-1}) \\
& \leq M _{l ; N N '} (\bt ^* \rho (a b ^{-1} ) )^N (\bs ^* \rho (a b ^{-1} )) ^{N '}
\prod _{i=1 } ^l (|X |_i + |Y |_i ),
\end{align*}
for some constants $M _{l ; N N'} > 0$. 
Note that any functions on $\mathcal S (\cG _0 )$ extends to a smooth function on $\cG$ by 0.
With such identification, we have for any $\varepsilon > 0$,
$$\Psi ^{- \infty } _{\varepsilon ; \infty } (\cG ) \subset \mathcal S (\cG _0 ).$$

\begin{thm}
\label{GenInv}
There exists $\varTheta \in \Psi ^{[- m]} (\cG ) \oplus \mathcal S (\cG _0 ) $ 
such that $G = \nu (\varTheta ) $.
\end{thm}

Let $ \kappa _1 , \kappa _2 $ be any two kernels in $ \Psi ^{- \infty } _{\varepsilon ; \infty } (\cG ) $.
Observe that the vector representations $\nu (\kappa _1 ) $ on $\cL ^2 (\rM _0 )$ 
is just convolution with $ \kappa _1 |_{\cG _0 }$.
Using the identification $\cG _0 \cong \rM _0 \times \rM _0 $,
one can write
$$ \nu (\kappa _1 ) f (x) = \int \kappa _1 |_{\cG _0 } (x, y) f (y) d y .$$
Observe that for any $y \in \rM _0 $,
$\kappa _1 | _{\bs ^{-1} (y )} = \kappa _1 (\cdot , y ) \in \cL ^2 (\rM _0 )$. 

Let $T : \cL ^2 (\rM _0 ) \to \cL ^2 (\rM _0 ) $ be any bounded linear map.
We claim that
\begin{lem}
\label{KerReg}
The function $F : \rM _0 \times \rM _0 \to \bbC $
$$ F(x, y ) := \int _{y _1 \in \rM _0 } 
\kappa _1 |_{ \bt ^{-1} (x) } (y _1 ) (T \kappa _2 | _{ \bs ^{-1} (y)}) (y _1 ) \mu _0 (y _1 )$$
is bounded and continuous.
\end{lem}
\begin{proof}
With $l = 1$, and restrict to $\cG _0 \cong \rM _0 \times \rM _0 $,
Definition \ref{FFDegen} reduces to
$$ d \kappa _2 (x, y) (X \oplus Y ) \leq M _1 e ^{- \varepsilon _2 d (x , y)} (|X| + |Y|),$$
for any vector $X \oplus Y \in T _{(x, y)} (\rM _0 \times \rM _0 )$.
For any $y' \in \rM _0 $ such that $t _0 := d (y, y')$ is smaller than the injectivity radius of $\rM _0 $,
let $\gamma (t)$ be the unique geodesic joining $y $ and $y'$.
Then 
$$ |\kappa _2 (y _2 , y') - \kappa  _2 (y _2 , y) | \leq M _1 e ^{- \varepsilon _2 (d (y _2 , y) - t _0 )} t _0 .$$
Regard $\kappa _2 (\cdot , y' ) - \kappa _2 ( \cdot , y) $ as a function in $\cL ^2 (\rM _0 )$ for each $y , y'$.
By the boundedness of $T$, there is some $K > 0$ such that 
\begin{align*}
\| T _{y _2} (\kappa _2 (y _2 , y' ) - \kappa _2 (y _2 , y)) (\cdot ) \| ^2 _{\cL ^2 (\rM _0 )}
\leq & K ^2 \| (\kappa _2 (\cdot , y' ) - \kappa _2 ( \cdot , y))\| ^2 _{\cL ^2 (\rM _0 )} \\
= K ^2 \int |\kappa _2 (y _1 , y') - \kappa _2 (y _1 , y) |^2 \mu _0 (y _1 )
\leq & K ^2 M _1 ^2 t _0 ^2 e ^{2 \varepsilon _2 t _0 } \int e ^{- 2 \varepsilon _2 d (y _1 , y)} \mu _0 ( y _1 ).
\end{align*}
Since $\cG$ has polynomial growth by our assumption, 
the last integral $ \int e ^{- 2 \varepsilon _2 d (y _1 , y)} \mu _0 ( y _1 )$ is bounded by some constant $C _2$,
independent of $y$.
Similarly, 
$$ \int |\kappa _1 (x' , y _1 ) - \kappa _1 (x , y _1 ) |^2 \mu _0 (y _1 )
\leq M^{\prime 2} _1 t _0 ^2 e ^{2 \varepsilon _1 t _0 } C _1.$$
for some constants $M' _1 , C _1 > 0$. It follows that
\begin{align*} 
F (x' , y') - F (x, y) 
= & \int (\kappa _1 (x' , y _1 ) - \kappa _1 (x , y _1 ) )
( T _{y _2} (\kappa _2 (y _2 , y') - \kappa _1 (y _2 , y)))(y _1) \mu _0 (y _1) \\
&+ \int \kappa _1 (x , y _1 ) ( T _{y _2} (\kappa _2 (y _2 , y') - \kappa _1 (y _2 , y)))(y _1) \mu _0 (y _1) \\
&+ \int (\kappa _1 (x' , y _1 ) - \kappa _1 (x , y _1 ) ) (T _{y _2} \kappa _2 (y _2 , y))(y _1) \mu _0 (y _1) \\
\leq & ( M' _1 t _0 e ^{\varepsilon _1 t _0 } \sqrt {C _1 } )(M _1 t _0 e ^{\varepsilon _2 t _0 } \sqrt {C _2} )
+ \| \kappa _1 (x , \cdot ) \| _{\cL ^2 } (M _1 t _0 e ^{\varepsilon _2 t _0 } \sqrt {C _2} ) \\
&+ (M _1 ' t _0 e ^{\varepsilon _1 t _0 } \sqrt {C _1} ) \| \kappa _2 (\cdot , y) \| _{\cL ^2}.
\end{align*}
It is clear that given any $(x, y) \in \rM _0 \times \rM _0 $, the right hand side goes to 0 as $t _0 \to 0$.
Hence $F (x, y)$ is continuous.

The proof of the boundedness of $F$ is similar.
We have $ | \kappa _1 (x, y _1) | \leq M '_0 e ^{- \varepsilon _1 d (x , y _1)} $, 
and $ | \kappa _2 (y _2, y ) | \leq M _0 e ^{- \varepsilon _1 d (y _2 , y)} $,
for some constants $M_0 , M' _0 $. 
Therefore
\begin{align*}
|F (x, y)| \leq & \| \kappa _1 (x, \cdot ) \| _{\cL ^2 (\rM _0 )} 
\| T _{y _2} \kappa _2 (y _2 , y) \|_{\cL ^2 (\rM _0 )} \\
\leq & \| \kappa _1 (x, \cdot ) \| _{\cL ^2 (\rM _0 )} 
K \| \kappa _2 (\cdot , y) \|_{\cL ^2 (\rM _0 )} \leq K M '_0 M _0 \sqrt {C _1 C _2}.
\end{align*}
\end{proof}

We turn to the proof of Theorem \ref{GenInv}.
\begin{proof}
(See \cite[Theorem 4.20]{Mazzeo;EdgeRev})
Let $\bar \varPhi $ be defined in Equation (\ref{InftyPara}) and $\tilde \varPhi $ be as in Remark \ref{InftyParaRem}.
Also, denote
$$ \tilde R := \id - \tilde \varPhi \varPsi , \bar R := \id - \varPsi \bar \varPhi .$$
Computing $\tilde \varPhi \varPsi G$ and $G \varPsi \bar \varPhi $ in two different ways, 
one gets the equality
$$ G = \tilde \varPhi |_{\cG _0 } + \tilde R |_{\cG _0 } G - \tilde \varPhi |_{\cG _0 } \circ P _\bot
= \bar \varPhi |_{\cG _0 } + G  \tilde R |_{\cG _0 } - P _0 \tilde \varPhi |_{\cG _0 } .$$
Rearranging, one gets
\begin{equation}
\label{Recurrent2}
G = \tilde \varPhi |_{\cG _0} + \tilde R |_{\cG _0 } G \bar R |_{\cG _0 }
+ \tilde R |_{\cG _0 } \bar \varPhi |_{\cG _0 } 
- \tilde R |_{\cG _0 } P _\bot \bar \varPhi |_{\cG _0 }
- \tilde \varPhi |_{\cG _0 } P _0 .
\end{equation}
It is straightforward to see that 
$$ \tilde R |_{\cG _0 } P _\bot \bar \varPhi |_{\cG _0 } \text{ and }
\tilde \varPhi |_{\cG _0 } P _0 \in \mathcal S (\cG _0 ).$$
It remains to consider $\tilde R |_{\cG _0 } G \bar R |_{\cG _0 }$.
From Lemma \ref{KerReg}, 
it follows that $\tilde R |_{\cG _0 } G \bar R |_{\cG _0 }$ 
is given by convolution with some bounded continuous kernel $\phi$ on $\rM _0 \times \rM _0 $.
Using Equation (\ref{Recurrent2}) again, one gets
\begin{align*}
G =& \: \tilde \varPhi |_{\cG _0} + \tilde R |_{\cG _0 }
\big(\tilde \varPhi |_{\cG _0} + \tilde R |_{\cG _0 } G \bar R |_{\cG _0 }
+ \tilde R |_{\cG _0 } \bar \varPhi |_{\cG _0 } 
- \tilde R |_{\cG _0 } P _\bot \bar \varPhi |_{\cG _0 } 
- \tilde \varPhi |_{\cG _0 } P _0  \big) \bar R |_{\cG _0 } \\
&+ \tilde R |_{\cG _0 } \bar \varPhi |_{\cG _0 } 
- \tilde R |_{\cG _0 } P _\bot \bar \varPhi |_{\cG _0 }
- \tilde \varPhi |_{\cG _0 } P _0 .
\end{align*}
It is clear that $G - \tilde \varPhi |_{\cG _0} \in \mathcal S (\cG _0 )$. 
Hence $ G - \tilde \varPhi |_{\cG _0} = \varTheta _{- \infty } |_{\cG _0 } $ 
for some pseudo-differential operator $\varTheta _{- \infty}$ of order $- \infty $ on $\cG$.
Finally we conclude that
$ G = \nu ( \tilde \varPhi + \varTheta _{- \infty })$, 
and $\tilde \varPhi + \varTheta _{- \infty } \in \Psi ^{[- m]} (\cG ) \oplus \mathcal S (\cG _0 )$.
\end{proof}

\subsection{The general case}
To describe the inverse of a uniformly supported elliptic pseudo-differential operator on a general uniformly degenerate
boundary groupoid $\cG = \bigsqcup _{k = 0 } ^r \cG _k \times \rM _k \times \rM _k $,
one repeats the arguments of Lemma \ref{InftyParaLem} and Theorem \ref{InvDecay}.
More precisely, it suffices to prove that
\begin{thm}
Let $\cG = \bigsqcup _{k = 0 } ^r \rG _k \times \rM _k \times \rM _k $ be a uniformly degenerate boundary groupoid
with smooth extension property. 
Given a uniformly supported elliptic pseudo-differential operator $\varPsi = \{ \varPsi _x \} $ such that
$\varPsi |_{\cG _k } $ are invertible for all $k \geq 1 $.
Suppose that for some $r' \leq r $ there $r - r'$ kernels $\{ \varphi ^{(k)} \}$, 
$k = r', \cdots, r $, of the form
$ \varphi ^{(k)} \sim \sum _{i = 1} \varphi ^{(k)} _i ,$
such that 
$$\id - \psi \circ (\varphi + \varphi ^{(r')} + \cdots + \varphi ^{(r)} ) 
\in \Psi ^{- \infty } _{\bullet ; \bzero _{r '} , \infty } (\cG ),$$
where
$\varphi ^{(k)} _1 \in C ^\infty (\cG ) \bigcap \Psi ^{- \infty } _{\bullet ; \bzero _{k + 1} , \infty } (\cG )$, 
and $\varphi ^{(k)} _i \in \Psi ^{- \infty } _{\bullet ; \bzero _{k } , \lambda ^{(k)} _i , \infty } (\cG )$,
$\varphi ^{(k)} _i |_{\cG \setminus \bar \cG _k } \in C ^\infty (\cG \setminus \bar \cG _k )$,
for all $i \geq 2$.
Then
\begin{enumerate}
\item
$\varPsi ^{-1} |_{\bar \cG _k } \in \Psi ^{[- m ]} (\bar \cG _k ) \oplus \Psi ^{- \infty } _\bullet (\bar \cG _k)$;
\item
There exists $\varphi ^{(r' )} \sim \sum _{i = 1} \varphi ^{(r')} _i $
where $\sum _{i = 1} \varphi ^{(r' - 1)} _i$ satisfy the same smoothness and decaying conditions as above,
such that
$$\id - \psi \circ (\varphi + \varphi ^{(r' - 1)} + \cdots + \varphi ^{(r)} ) 
\in \Psi ^{- \infty } _{\bullet ; \bzero _{r' - 1} , \infty } (\cG ).$$
\end{enumerate}
\end{thm}
\begin{proof}
To prove claim (i), 
let $\phi _k $ be the kernel of $( \varPsi |_{\bar \cG _k } )^{-1}$
and consider $ \phi _k \circ (\psi  \circ(\varphi + \varphi ^{(r')} + \cdots + \varphi ^{(r)} ) )|_{\bar \cG _k}$
using the same arguments as Theorem \ref{InvDecay}. 
Moreover, by Equation (\ref{InvRecursion}), 
$$ \phi _k - (\varphi + \varphi ^{(r')} + \cdots + \varphi ^{(r)} )|_{\bar \cG _k}
\in \Psi ^{- \infty } _{\bullet ; \infty } (\bar \cG _k ).$$
Using the smooth extension property, let 
$\hat \varphi _{r' - 1} \in \Psi ^{- \infty } _{\bullet ; \bzero _{r' - 1 }, \infty } \bigcap C ^\infty (\cG )$
be such that
$$ \id - \psi \circ (\varphi + \hat \varphi _{r' - 1} + \varphi ^{(r')} + \cdots + \varphi ^{(r)} ) )
\in \Psi ^{- \infty } _{\bullet ; \bzero _{r' - 1}, \lambda , \infty } (\cG ). $$
Then the arguments of Lemma \ref{InftyParaLem} can be applied to prove (ii).
Note that the Neumann series is finite on compact subsets and hence the limit is in 
$C ^\infty (\cG \setminus \bar \cG _{k-1})$.
\end{proof}


\section{Concluding remarks}
In this paper, we constructed a rather complete analogue of the big and full calculus to \cite{Mazzeo;EdgeRev}, 
namely, the exponentially decaying calculus and a finer space of kernels with asymptotic expansions.
We proved that these spaces are filtered like the full calculus, 
and contains the compact parametrices and generalized inverse of elliptic differential operators.

We remark that the definition of boundary groupoids and uniformly degenerate operators 
we considered is somewhat restricted.
For instance, it would seems to be rather obvious to generalize to the notion of boundary groupoids to contain
invariant sub-manifolds of the form $\rG _k \times \rM \times _{\mathrm B} \rM $.
Also, proving conjecture \ref{MainConj} would be a major advancement of the theory.

The full calculus constructed in this paper should enable one to re-write many classical results in the groupoid context.
On the more geometrical side, some construction had been exemplified in \cite{So;PhD}.
There, the author considers the heat kernel of generalized 
Laplacian operators and constructs renormalized index for the Bruhat sphere.
One should be able to generalize the results in \cite{So;PhD} with the framework constructed here.
In particular, the functions $\rho _k$ can be used as regularizing functions.
In the same vein, complex powers of elliptic operators, 
as well as holomorphic functional calculus of groupoid pseudo-differential operators, 
are also very interesting directions for future research. 

On the side of more traditional analysis, 
one would study boundary problems involving (vector representations of) groupoid differential operators,
or even non-linear equations.

\appendix
\section{Manifolds with bounded geometry}
In this section, 
we recall the definition of manifolds with bounded geometry and some classes of functions and operators defined it.
For details, see \cite{Shubin;BdGeom}.
\begin{dfn}
\label{UBDfn}
A Riemannian manifold $\rM$ is said to have bounded geometry if
\begin{enumerate}
\item
$ \rM $ has positive injectivity radius;
\item
The Riemannian curvature $R$ of $\rM $ has bounded covariant derivatives.
\end{enumerate}
\end{dfn}

\begin{lem}
\label{BallCover}
\cite[Lemma 1.2]{Shubin;BdGeom}
There exists $r _0 > 0$ such that for any $0 < r < r _0 $,
there is a countable set $\{ x_\alpha \} \subset \rM$ such that the balls 
$B (x_\alpha , \varepsilon )$ is a cover of $\rM$,
and any $x \in \rM$ belongs to at most $N$ balls $B (x_\alpha , 2 r ) $, for some $N$ independent of $x$.
\end{lem}

\begin{lem}
\label{BddPartition}
Let $\{ (B ( x _\alpha , \varepsilon ) , \bx _\alpha ) \} $ be a cover by normal coordinates patches,
such that the conclusion of Lemma \ref{BallCover} holds.
Then there exists a partition on unity $\theta _\alpha $ subordinated to $\{ B (x_\alpha , \varepsilon ) \}$,
such that for any $k \in \bbN$, 
all $k$-th order partial derivatives of $\theta _\alpha $ are bounded by some $C_k$, independent of $\alpha $.
\end{lem}

For each $m \in \bbR$, define the 2-norms
\begin{equation}
\label{UBSobo}
\| f \| _{2, m} := \Big( \sum _{\alpha } 
\| \theta _\alpha f \| _{\cW ^m (U _\alpha ) } ^2 \Big)^{\frac{1}{2}},
\end{equation}
where $ \cW ^m (U _\alpha ) $ is the $m$-th Sobolev norm on $U _\alpha \subset \bbR ^n$.
We denote the completion of $C ^\infty _c (\rM _0) $ with respect to 
$\| \cdot \| _{2 , m}$ by $\cW ^{m} (\rM )$.

Observe that, since all transition functions are uniformly bounded, 
the equivalence classes of these norms are independent of the choices made. 

On a manifold with bounded geometry, a class of `uniformly bounded' pseudo-differential operators can also be defined.
Fix any covering $\{ U _\alpha , \bx _\alpha \}$ of $\rM$ by normal coordinates.
Let $\varPsi \in \psi ^m _\varrho (\rM)$.
Recall that $( \bx _\alpha ^{-1} )^* \psi \bx _\alpha ^* $ is a pseudo-differential operator on $U _\alpha $.
Let $\sigma _\alpha \in \mathbf S ^m (U _\alpha )$ be the total symbol of $( \bx _\alpha ^{-1} )^* \psi \bx _\alpha ^* $.
Then we say that 
\begin{dfn}
The pseudo-differential operator $\varPsi $ is {\it uniformly bounded} if 
\begin{enumerate}
\item
The support of $\varPsi $ is contained in the set
$$ \{ (x , y ) \in \rM \times \rM : d (x , y ) < r \} $$
for some $r > 0$;
\item
For any multi-indexes $I , J$,  
there exists a constant $C _{I J} $, independent of $\alpha $,
such that 
$$ | \partial _x ^I \partial _\zeta ^J \sigma _\alpha | \leq C _{I J} (1 + |\zeta |)^{m - |J|} .$$ 
\end{enumerate}
We denote the set of all, 
uniformly bounded pseudo-differential operators of order $\leq m$ by $\Psi ^m _b (\rM)$.
\end{dfn}

\section{Proof of Theorem \ref{SimpleExt}}
We consider the special case when $\cG = \rM _0 \times \rM _0 \bigsqcup \rG \times \rM _1 \times \rM _1 $.
Let $p = \dim \rG , q = \dim \rM _1 $.
For simplicity, we denote the only defining function by $\rho$.

\subsection{The exponential map}
First, recall the definition of admissible section and exponential map of a groupoid.
\begin{dfn}
An admissible section is a smooth map $S : \rM \to \cG $ such that
$ \bs \circ S = \id $ and $ \bt \circ S $ is a diffeomorphism on $\rM $.
\end{dfn}
One has a semi-group structure on the set of all admissible sections defined by
$$ S _1 S _2 (x) := S _1 (\bt \circ S _2 (x)) S _2 (x), $$
where the right hand side is the groupoid multiplication. 
Likewise, each admissible section $ S $ induces a diffeomorphism on $\cG $ given by
$$ a S := a S ((\bt \circ S ) ^{-1} (a)) .$$
It is easy to see that $ ( a S _1 ) S _2 = a (S _1 S _2) $ for any admissible sections $S _1 , S _2 $.

\begin{rem}
In the special case when $\cG = \rG $ is a Lie group,
$ Z \mapsto \exp Z (e) $ is just the Lie group exponential map.
\end{rem}

Given any smooth section $X \in \Gamma ^\infty (\rA)$,
denote by $X ^\br$ the right invariant vector field on $\cG $ with $\bs ^* X ^\br = 0 $ and $X ^\br |_\rM = X $.
Since $\rM $ is compact, 
it is standard that $X　^\br$ is a complete vector field on $\cG$,
hence one has a well defined map 
$$ \exp X : \rM \to \cG ,$$
given by the flow of $X ^\br $ form each $x \in \rM \subset \cG $.
It is a well known fact that $\bt \circ \exp X  $ equals the flow of 
$ \nu (X) $ on $\rM $ and hence is a $\exp X $ is an admissible section.
Define
$$ E _X := d \bt \circ d ( \exp X |_ \rA ) : \rA \to \rA .$$

We list some basic properties of the exponential map \cite{Nistor;IntAlg'oid}, \cite{Mackenzie;Book}:
\begin{enumerate}
\item
For any $X , Y \in C^\infty (\rA ), \exp X \exp Y = \exp Y \exp E _X $;
\item
For any $x \in \rM , ((\exp X )(x))^{-1} = \exp (- X) (E ^\nu _X (x)) $,
where $E ^\nu _X : \rM \to \rM $ is the flow of $\nu (X)$.
\end{enumerate}

\begin{nota}
For any collection of sections $Z _I = (Z_1 , \cdots Z _{|I|} ) \in \Gamma ^\infty (\rA )$, denote 
$$ \exp Z _I := \exp Z _{|I|} \exp Z _{|I| - 1} \cdots \exp Z _2 \exp Z _1 .$$
For any $\mu = (\mu _1 , \cdots , \mu _{|I|}) \in \bbR ^{|I|} $, denote
$$ \exp ( \mu \cdot Z _I )
:= \exp \mu _{|I|} Z _{|I|} \exp \mu _{|I| - 1} Z _{|I| - 1} \cdots \exp \mu _2 Z _2 \exp \mu _1 Z _1 .$$
\end{nota}

We adapt the construction of exponential coordinates charts on a groupoid in \cite{Nistor;IntAlg'oid} to our case.
\begin{lem}
\label{ExpCoord}
Let $Z _I \subset \Gamma ^\infty (\rA ) , U _\alpha $ coordinates patch of $\rM $.
Let $X ^{(\alpha )} _1 , \cdots , X ^{(\alpha )} _k $ be a local basis over $ ( \bt \circ \exp Z _I )(U _\alpha ) $.
Then there exists $\delta > 0 $ such that the map
$\bx ^{(\alpha )} _{Z _I } : ( - \delta , \delta ) ^k \times U _\alpha \to \cG $,
$$ \bx ^{(\alpha )} _{Z _I } (\tau , x)
:= \exp ( \tau \cdot ( X ^{(\alpha )} _1 , \cdots , X ^{(\alpha )} _k ) ) \exp Z _I (x), $$
is a diffeomorphism onto its image.
\end{lem}

\subsection{Exponential coordinates on $(\rM _0 \times \rM _0 ) \bigsqcup (\rG \times \rM _1 \times \rM _1 )$}
We turn to our special case when $\cG = (\rM _0 \times \rM _0 ) \bigsqcup (\rG \times \rM _1 \times \rM _1 )$.
First consider exponential coordinates on $\rG $.
\begin{lem}
\label{GCover}
Let $(Y _1 , \cdots , Y _p )$ be a fixed basis of $\mathfrak g $.
There exists a cover of $\rG $ by coordinates patches of the form 
$$ (- \delta , \delta ) ^p \ni \mu \mapsto \exp (\mu \cdot (Y _1 , \cdots , Y _p )) \exp (Z ^\rG _I )(e) ,$$
for some collections $Z ^\rG _I \in \mathfrak g $, such that
\begin{enumerate}
\item
The cover is locally finite with uniformly bounded index;
\item
$ |Z ^\rG _1 |, \cdots , | Z ^\rG _ {|I|} | \leq r _\fg $,
for some $r _\fg > 0$ (independent of $I$);
\item
There exist a constant $ C _\rG > 0 $ such that
$$ d (e , \exp Z ^\rG _I ) > C _\rG (|I| - 1), $$
where $d (\cdot , \cdot ) $ is the right invariant metric on $\rG$;
\end{enumerate}
\end{lem}
\begin{proof}
For any $r > 0$, denote by $B _\rG (g, r) $ and $B _{\mathfrak g } (0 , r) $ 
the ball on $\rG $ (resp. $\mathfrak g $) of radius $r $ centered at $g \in \rG $ (resp. $0 \in \mathfrak g $).

Let $ r > 0 $ be such that
$ B (e , 2 r ) \subseteq 
\{ \exp (\mu \cdot ( Y _1 , \cdots , Y _p )(e)) : \mu \in ( - \delta , \delta ) ^p \} .$
Take a maximal collection of subset of the form 
$$ B _\rG ( g _i , r ) = \{ g g _i : g \in B _\rG (e , r ) \} .$$
Since $\rG $ is a manifold with bounded geometry, 
it is standard that $ \{ B _\rG ( g _i , 2 r ) \} $ is a covering satisfying condition (i), 
and hence the covering 
$$\{ \exp (\mu \cdot ( Y _1 , \cdots , Y _p )(e)) g _i : \mu \in ( - \delta , \delta ) ^p \}.$$

It remains to find for each $i$, $g _i = \exp Z ^\rG _{I _i } $, some $ Z ^\rG _ {I _i }$ satisfying condition (ii).
It is elementary that there exists a constant $ r _\fg > 0$ such that the exponential map
$$ \exp : B _{\mathfrak g } (0 , r _\fg ) \to \rG $$ 
is a diffeomorphism onto its image. 
Moreover, one has $ \exp  ( B _{\mathfrak g } (0 , r _\fg )) \supseteq B _\rG (e , C _\rG )$
for some constants $C _\rG > 0$.

Let $ \gamma (t) $ be a unit speed minimizing geodesic joining $e $ and $g _i $.
Parameterize $\gamma $ so that $ \gamma (0 ) = e , g (L) = g _i$.
Then $L = d (g _i , e ) $.
Define $g _l := \gamma ( C _\rG l) , l = 0 ,1 , \cdots L' $, 
where $ L' $ is the largest integer such that $C _\rG L' \leq L $.
Then $g = g g _{L' } ^{-1} g _{ L'} \cdots g _1 g _0 ^{-1} g _0 $.
By right invariance $g _I g _{L '} ^{-1} , g _l g _{l -1 } ^{-1} \in B _\rG (e , r _\fg) $ for any $l$.
Therefore by definition there exists $Z _l \in \mathfrak g , 1 \leq l \leq L' + 1$,
such that $|Z ^\rG _l| < r _\fg $ and 
$$\exp Z ^\rG _{L' + 1} = g g ^{-1} _{L ' } , \quad \exp Z ^\rG _l = g _l g _{l - 1} ^{-1} \quad \Forall l \leq L'.$$
Let $ Z ^\rG _{I _i } := \{ Z ^\rG _1 , \cdots , Z ^\rG _{L' + 1} \}$.
It is clear that the collection $Z ^\rG _{I _i} $ satisfies conditions (ii) and (iii).    
\end{proof}

Let $\{ Y ^{\mathfrak g} _1 , \cdots Y ^{\mathfrak g} _ p \}$ be an orthonormal basis of $\mathfrak g$.
Regard $\{ Y ^{\mathfrak g} _1 , \cdots Y ^{\mathfrak g} _ p \}$ as a basis of $\rM _1 \times \mathfrak g \to \rM _1$ .
Let $\{ Y _1 , \cdots Y _p \}$ be an extension of 
$\{ Y ^{\mathfrak g} _1 , \cdots Y ^{\mathfrak g} _ p \}$ to $\Gamma ^\infty (\rA )$, 
such that $\{ Y _1 , \cdots Y _p \}$ is a local orthonormal basis on some open set $ U \supset \rM _1$ of $\rM$.

\begin{lem}
\label{YCover}
Given any collection $ Z _I = (Z _1 , \cdots , Z _{|I|} ) \subset \Span _\bbR \{Y _1 , \cdots Y _p \} $,
$ |Z _m | < r _\fg $ for all $1 \leq m \leq |I| $.
For any $M > 0$, there exists $r > 0 $ such that 
$$ d (x , \bt \circ \exp Z _I (x) ) \leq M ,$$
whenever $x \in B (\rM _1 , e ^{- \omega r _\fg |I|} r ) $.
\end{lem}
\begin{proof}
Let $\rho $ be a smooth function on $\rM \setminus \rM _1 $, such that $\rho = d (\cdot , \rM _1 )$,
as in Lemma \ref{LocalDegen}. 
Since for any $Z \in \Span _\bbR \{Y _1 , \cdots Y _p \} $,
$\nu (Z) = 0 $ on $\rM _1 $ and $ | \nu (Z ) | $ is a Lipschitz function,
there exists $M' > 0 $ such that 
\begin{equation}
\label{LengthEst1}
|\nu (Z ) (x) | \leq M' \rho (x) ,
\end{equation}
for any $Z \in \Span _\bbR \{Y _1 , \cdots Y _p \} , |Z| \leq r _\fg .$

Write $x _0 := x, x _1 := \exp Z _{i - 1} \cdots \exp Z _1 (x), i =1, \cdots |I|$.
Then $x _{i -1 } , x _{i}$ is joined by the curve $(\bt \circ \exp t Z _i )(x _{i-1}), t \in [0, 1] $,
whose length is 
\begin{align*}
\int _0 ^1 |\nu (Z _i ) ((\bt \circ \exp t Z _i )(x _{i - 1})) | d t
\leq & \int _0 ^1 M' \rho ((\bt \circ \exp t Z _i )(x _{i - 1})) d t \\
\leq & \int _0 ^1 M' e ^{ \omega r _\fg (i + 1) } \rho (x) d t,
\end{align*}
where we used Theorem \ref{DEst} for the last inequality.
Hence by the triangular inequality 
$$ d (x , \bt \circ \exp Z _I (x) ) 
\leq \frac{M' e ^{ \omega r _\fg } (e ^{ \omega r _\fg |I|} - 1) \rho (x)}{e ^{ \omega r _\fg } - 1 }.$$
The claim then follows by putting $r _0 \leq \frac{M (e ^{ \omega r _\fg } - 1)}{M ' e ^{ \omega r _\fg }} $ 
and such that $\rho = d (\cdot , \rM _1 ) $ on $B ( \rM _1 , r _0 )$.
\end{proof}

Let $L > 0 $ be such that the injectivity radius of $\rM _1 $ is greater than $2 L $.
Then $\rM _1 $ can be covered by a finite collection of balls $\{ B _{\rM _1} ( x _\alpha , L ) \}$.
Let $\bx _{\rM _1} ^{(\alpha )}$ be a local coordinates chart of $ B _{\rM _1} ( x _\alpha , 2 L )$.
Fix a trivialization of $T \rM _1 ^\bot |_{ B _{\rM _1} ( x _\alpha , 2 L )}$ for each $\alpha $ 
and let $\tilde U ^\alpha $ be the coordinate patches 
$$ \exp _{\bx _{\rM _1} ^{( \alpha )} (x _1 , \cdots x _q)} (x _{p + 1} , \cdots , x _n ) ,$$ 
where $\exp$ here denotes the Riemannian exponential map.

Fix an orthonormal basis $\{ Y _1 , \cdots , Y _p \} $ of $\fg$.
Regard it as a basis of $ T \rM _1 \times \fg$ and extend to an orthonormal set of sections on 
$\rA |_{\bigcup \tilde U _\alpha }$. We still denote the extension by $\{ Y _1 , \cdots , Y _p \} $.
It is then a standard construction that there exists 
\begin{itemize}
\item
A finite set of collections of sections $ X _J \subset \Gamma ^\infty (\rA)$;
\item
for each $ \alpha , J $, a basis
$X ^{\alpha , J} _1 , \cdots , X ^{\alpha , J} _p \in \Gamma ^\infty ( \exp X _J (\tilde U _\alpha ) )$, 
\end{itemize}
such that
\begin{enumerate}
\item
the local coordinates
$ \tau \mapsto \exp (\tau \cdot (X ^{\alpha , J} _1 , \cdots , X ^{\alpha , J} _p )) \exp X _J (x) , $
$\tau \in (- \delta , \delta ) ^p , \\ x \in U _\alpha \bigcap \rM _1 $, is an atlas of $\rM _1 \times \rM _1 $;
\item
On $\exp X _J (\tilde U _\alpha )$, $\{ Y _1 , \cdots Y _p , X ^{\alpha , J} _1 , \cdots , X ^{\alpha , J} _q \}$
is an orthonormal basis of $\rA |_{\exp X _I (\tilde U _\alpha )} $.
\end{enumerate}
It follows that for some $r > 0$, the map 
$\bx ^\alpha _{Z _I , X _{J}} : (- \delta , \delta ) ^{p+q} 
\times (U _\alpha \bigcap B (\rM _1 , e ^{-\omega r _\fg |I|} r )) \to \cG $ defined by
\begin{align}
\label{ExpCover}
\bx ^{(\alpha )}_{X _{J} , Z _I　} & (\mu , \tau , x) \\ \nonumber
&　:= \exp (\mu \cdot ( Y _1 , \cdots , Y _p )) 
\exp (\tau \cdot (X ^{\alpha , I} _1 , \cdots , X ^{\alpha , I} _p )) \exp X _{J} \exp Z _I (x)
\end{align}
satisfies the conditions of Lemma \ref{ExpCoord} and hence defines a local coordinates patch in $\cG$.
Moreover, $\{ U ^\alpha _{X_J , Z_I } \bigcap (\rG \times \rM _1 \times \rM _1 ) \}$ is an atlas of 
$\rG \times \rM _1 \times \rM _1$.

Recall that $\tilde \cG := \{ (a , b) \in \cG \times \cG : \bs (a ) = \bs (b) \} $
and one has the map $\tilde \bm : \tilde \cG \to \cG $ defined by $\tilde \bm (a , b ) := a b ^{-1 } $.
Consider writing $d \tilde \bm (V ^\br \oplus W ^\br ) $ on the coordinates chart 
$ ( \bx ^{(\alpha )} _{X _{J} ,Z _I} , U ^{(\alpha )} _{X _{J} ,Z _I} ). $

By Equation (\ref{ReducedVF}), it is straightforward to compute for any $ a \in U ^{(\alpha )} _{X _{J} , Z _I}$,
$$ \partial _1 (a) = Y _1 ^\br (a), \partial _2 = ( E _{Y _1} Y _2 )^\br (a) , \cdots $$
and so on. It follows that on the coordinates chart
$$ \bx ^{(\alpha )} _{\emptyset} (\mu ' , \tau ' , x')
:= \exp (\tau \cdot (X ^{\alpha , I} _1 , \cdots , X ^{\alpha , I} _p )) 
( E _{Z _I } ^{-1} \circ E _{X _J} ^{-1} (x)), $$
if one writes 
$ V ^\br = \sum _{i = 1 } ^n v _i (x') \partial _{(\mu' , \tau ') _i}$
on $ U ^{(\alpha )} _{\emptyset}$ 
(Note that there is no $\partial _{x' _i} $ since $V ^\br $ is tangential to the $\bs$-fibers ),
then $V ^\br = \sum _{i = 1} ^n v _i (E _{X _J} E _{Z _I} (x)) 
\times \partial _{(\mu , \tau )_i }$ on $ U ^{(\alpha )} _{X _J , Z _I}$.

We turn to consider the case $d \tilde \bm (0 \oplus W ^\br ) $.
For any
$ a = \bx ^{(\alpha )} _{X _{J} ,Z _I} (\tau , \mu , x ) ,$
one has
\begin{align} 
\label{XInverse}
& d \tilde \bm (0 \oplus W ^\br ) (a) \\ \nonumber
=& \partial _t \Big|_{t = 0 } 
\exp (- t E _{\mu _1 Y _1 } \circ \cdots \circ E _{\mu _p Y _p } 
\circ E _{ \tau _1 X ^{(\alpha )} _1 } \circ \cdots \circ E _{ \tau _q X ^{(\alpha )} _q }
\circ E _{X _J} \circ E _{Z _1 } W ) \\ \nonumber
& \exp ( \tau \cdot ( X ^{(\alpha )} _1 , \cdots , X ^{(\alpha )} _q ) )\exp ( \mu \cdot (Y _1 , \cdots , Y _p ))
\exp {X _J} \exp {Z _I} \big(E ^\nu _{t W }(x) ) \big) \\ \nonumber
=& (E _{\mu _1 Y _1 } \circ \cdots \circ E _{\mu _p Y _p } 
\circ E _{ \tau _1 X ^{(\alpha )} _1 } \circ \cdots \circ E _{ \tau _q X ^{(\alpha )} _q }
\circ E _{X _J} \circ E _{Z _1 } W ) ^\br (a) \\ \nonumber
&+ d \big( \exp ( \tau \cdot ( X ^{(\alpha )} _1 , \cdots , X ^{(\alpha )} _q ) )
\exp ( \mu \cdot (Y _1 , \cdots , Y _p )) \exp _{X _J} \exp _{Z _I} \big) \nu (W)(\bs (a)).
\end{align}

To proceed, we estimate $E _{Z _I } W $.
\begin{lem}
\label{VGrowthEst}
There exists constants $K , r > 0 $, such that
for any $Z _I = \{ Z _1 , \cdots , Z _{I} \} \subset \Span _\bbR \{ Y _1 , \cdots Y _p \}, 
|Z _1| , \cdots |Z _{I}| \leq r _\fg$,
\begin{enumerate}
\item
$B (\rM _1 , r ) \subset \bigcup _{\alpha } U _{\alpha } $;
\item
For any $W \in \rA |_{B (\rM _1 , e ^{- \omega r _\fg |I|} r )}$
$ | E _{Z _{|I|} } \circ \cdots \circ E _{Z _1 } W | \leq K e ^{C (\log |I|)^2 } $,
\end{enumerate}  
where $\omega > 0 $ is such that Equation (\ref{LocalDegEq}) is satisfied on $B (\rM _1 , r)$.
\end{lem}
\begin{proof}
Only estimate (ii) is not obvious.
For each $\alpha $,
define $P ^\alpha : U _\alpha \to \rM _1 \bigcap U _\alpha  $ to be the coordinates projection.
For each $x \in \tilde U _\alpha $, define $ T ^\alpha _a : \rA |_{\tilde U _\alpha } \to \rA _{x }$
to be the natural projection by identifying $\rA |_{\tilde U _\alpha } \cong \bbR ^{p + q} \times \tilde U _\alpha $,
using the basis $\{ Y _1 , \cdots Y _p , X ^{(\alpha)} _1 , \cdots , X ^{(\alpha )} _q \}$. 

Define the functions
$ E ^\alpha : \Span _\bbR \{ Y _1 , \cdots Y _p \} \times \rA |_{U _\alpha } \to \rA |_{\tilde U _\alpha } $,
$$ E ^{\alpha } _{Z } W 
:= T ^\alpha _{\Phi ^\nu _Z (x)} \circ E _Z \circ T ^\alpha _{P ^\alpha (x)} (W), \quad W \in \rA _x , $$
and $F ^\alpha := E - E ^\alpha $.

We analyze $E ^ \alpha $. 
Let $W _0 := W \in \rA _{x _0 } $ and $W _i := E ^{\alpha }_{Z _i } \circ \cdots \circ E ^{\alpha } _{Z _1 } W 
\in \rA _{x _i}$.
Define $P _{ \mathfrak g } : \rA |_{\rM _1} \to \mathfrak g $ to be the natural projection, then one has
\begin{equation}
\label{IteratedAd}
P _{\mathfrak g } ( T ^{\alpha }_{P _{\alpha }} (x _{i + 1 } ) E ^{ \alpha } (W _i ))
= \Ad _{\exp Z _i } ( P _{ \mathfrak g } (T ^{\alpha } _{P ^\alpha (x _i)} (W _i))).
\end{equation}
Iterating Equation (\ref{IteratedAd}), one gets
\begin{equation}
\label{IteratedAd2}
P _{\mathfrak g } ( T ^{\alpha } _{P _{\alpha } (x _{m} )}(W _m ))
= \Ad _{\exp Z _m} \circ \cdots \circ \Ad _{\exp Z _1} P _{\mathfrak g } 
( T ^{\alpha } _{P _{\alpha } (x _0 )}(W )).
\end{equation}
Since $ \{ X^ {(\alpha )} , Y \} $ are orthonormal bases, $T ^\alpha _x $ is an isometry for any $x $.
Hence, using Equation (\ref{IteratedAd2}) and the fact that $E _{Z _i } $ acts as identity on $T \rM _1 $,
one gets
\begin{align*}
| W _m | =& | E ^{\alpha }_{Z _i } \circ \cdots \circ E ^{\alpha } _{Z _1 } W | \\
=& \big( \big|\Ad _{\exp Z _m} \circ \cdots \circ \Ad _{\exp Z _1} P _{\mathfrak g } 
( T ^{\alpha } _{P _{\alpha } (x _0 )}(W )) \big|^2 
+ \big|(\id - P _{\mathfrak g }) 
( T ^{\alpha } _{P _{\alpha } (x _0 )}(W )) \big| ^2 \big)^{\frac{1}{2}} .
\end{align*}
Using the assumption that $\rG $ is nilpotent, one can find a constant $ N _\rG $ such that
$$ |\Ad _{\exp Z' _{|I|}} \circ \cdots \circ \Ad _{\exp Z' _1} | \leq K _0 |I| ^ {N _\rG} , $$
for any collection $Z' _I = \{ Z' _1 , \cdots , Z ' _{|I|} \} \subset B $.
Then it is clear that
\begin{equation}
\label{IteratedEst}
| E ^{\alpha }_{Z _m } \circ \cdots \circ E ^{\alpha } _{Z _1 } W |
\leq K _1 m ^{N _\rG } | W |.
\end{equation}

We turn to $F ^\alpha $. Observe that $F _Z ^\alpha W = 0$ for any $\alpha , W \in \rA |_{\rM _1 } $.
Regard $F _Z ^ \alpha $ as a matrix valued function on $U _\alpha \bigcap B ( \rM _1 , r )$.
Then the differentiability of $F _Z ^\alpha $ implies there exists $K _2 > 0 $ such that
$$ | F ^ \alpha _Z W | \leq K _2 d (x , \rM _1 ) |W| , 
\quad \Forall W \in \rA _x , x \in U _\alpha \bigcap B ( \rM _1 , r ) .$$ 

Now we return to $ E _{ Z _m } \circ \cdots \circ E _{ Z _1 } W $.
We expand
\begin{align}
\label{BigSum}
E _{ Z _m } \circ \cdots & \circ E _{ Z _1 } W \\ \nonumber
=& E ^{\alpha }_{Z _m } \circ E^{\alpha }_{Z _{m -1} } \circ \cdots \circ E ^{\alpha } _{Z _1 } W
+ F ^{\alpha }_{Z _m } \circ E^{\alpha }_{Z _{m -1} } \circ \cdots \circ E ^{\alpha } _{Z _1 } W 
\\ \nonumber
&+ E ^{\alpha }_{Z _m } \circ F^{\alpha }_{Z _{m -1} } \circ \cdots \circ E ^{\alpha } _{Z _1 } W
+ F ^{\alpha }_{Z _m } \circ F^{\alpha }_{Z _{m -1} } \circ \cdots \circ E ^{\alpha } _{Z _1 } W 
\\ \nonumber
&+ \cdots 
+ F ^{\alpha }_{Z _m } \circ F^{\alpha }_{Z _{m -1} } \circ \cdots \circ F ^{\alpha } _{Z _1 } W .
\end{align}
From Equation (\ref{IteratedEst}) and our construction of $F ^\alpha $, 
it is straightforward to estimate that each term of the right hand side of Equation (\ref{BigSum}) is bounded by
$$ (K _2 r )^{m ' } e ^{- \frac{\omega r _\fg (m ^{\prime 2 } + m')}{2}} ( K _1 m ^{N _\rG } ) ^{m' + 1} ,$$
where $m' $ is the number of $F ^\alpha $.
One then adds all terms in the right hand side of Equation (\ref{BigSum}) together and gets the estimate
\begin{align*}
\label{SubExpSum}
\nonumber
| E _{Z _m } \circ \cdots \circ E _{Z _1 } W |
\leq & | W | \sum _{m ' = 0 } ^m \frac{m !}{m' ! (m - m')!} 
(K _2 r )^{m ' } e ^{- \frac{\omega r _\fg (m ^{\prime 2 } + m')}{2}} ( K _1 m ^{N _\rG } ) ^{m' + 1} \\
\leq & | W | \sum _{m ' = 0 } ^m m ^{m'} 
(K _2 r )^{m ' } e ^{- \frac{\omega r _\fg (m ^{\prime 2 } + m')}{2}} ( K _1 m ^{N _\rG } ) ^{m' + 1}.
\end{align*}
We split the above sum into two: the first form $m = 0 $ to $m' = N - 1$, and the second from $m ' = N $ to $m' = m$,
where $N$ is the smallest positive integer such that 
$$ (K _2 r ) e ^{- \frac{\omega r _\fg N }{2}} ( K _1 m ^{N _\rG + 1} ) < \frac{1}{2}, $$
in other words, $ N > \frac {2 \log (2 K _2 r K _1 m ^{ N _\rG + 1} )}{ \omega r _\fg }$.
Then one has
$$ | W | \sum _{m ' = N } ^m m ^{m'} 
K _1 m ^{N _\rG } ( K _2 r e ^{- \frac{\omega r _\fg (m ^{\prime } + 1)}{2}} K _1 m ^{N _\rG }) ^{m'} \\
\leq \frac{ |W| K _1 m ^{N _\rG } }{2 ^{N - 1 }},$$
by assumption. On the other hand, 
\begin{align*}
| W | \sum _{m ' = 0 } ^{N-1} & 
K _1 m ^{N _\rG } ( K _2 r e ^{- \frac{\omega r _\fg (m ^{\prime } + 1)}{2}} K _1 m ^{N _\rG + 1}) ^{m'} \\
\leq & | W | \sum _{m ' = 0 } ^{N-1} 
K _1 m ^{N _\rG } ( K _2 r e ^{- \frac{\omega r _\fg }{2}} K _1 m ^{N _\rG + 1}) ^{m'} \\
= & \frac{ |W| K_1 m ^{N _\rG } (( K _2 r e ^{- \frac{\omega r _\fg }{2}} K _1 m ^{N _\rG + 1}) ^N - 1)}
{( K _2 r e ^{- \frac{\omega r _\fg }{2}} K _1 m ^{N _\rG + 1}) - 1}.
\end{align*}
Observe that 
$$ m ^N \leq K _3 m ^{N' \log m } = K _3 e ^{N' (\log m )^2} ,$$
for some constants $K _3 , N'$. Therefore the estimation (ii) follows.
\end{proof}

\subsection{An exponentially decaying extension}
In this section, fix a coordinates cover as defined in Equation (\ref{ExpCover}). 
Let $\theta ^\rG _{Z _I } $ be a partition of unity of $\rG $ subordinated to $B ( \exp Z _I , r )$,
and $\theta ^\alpha _{X _{J}}$ be a partition of unity of $\rM _1 $ subordinated to $ U _\alpha \bigcap \rM _1 $.
Let $\theta \in C ^\infty _c (\bbR ) $ be such that $\chi $ equals 1 on $(- \infty , 1 )$ and 0 on $(2, \infty )$.

Given any $\psi \in \Psi ^\infty _\varepsilon (\cG _1 )$, 
define $ \theta ^\alpha _{X _{J} , Z _I} \in C ^\infty _c (U ^{(\alpha )} _{X _{J} , Z _I} )$ by
\begin{align*}
\theta ^\alpha _{X _{J} , Z _I} (\bx ^{(\alpha )} _{X _{J} , Z _I} (\tau , \mu ,  x))
:=& \; \theta ^\alpha _{X _{J}} (\exp ( \tau \cdot (X ^{\alpha , J} _1 , \cdots , X ^{\alpha , J} _p ))
\exp X _{I '}) (P ^\alpha (x)) \\
& \times \theta ^\rG _{Z _I} (\exp (\mu \cdot (Y _1 , \cdots , Y _p ) \exp Z _I (e))(P ^\alpha (x)) \\
& \times \theta (2 e ^{\omega r _\fg |I|} r ^{-1} \rho (x )).
\end{align*}
Here, recall that $P ^\alpha : U ^\alpha \to U ^\alpha \bigcap \rM _1 $ is the coordinates projection. 
Given any $\psi \in \Psi ^\infty _{\varepsilon ; \bzero } (\cG _1 ) , \varepsilon > 0$, let
\begin{equation}
\label{Part}
\bar \psi := \sum  _{X _{J} , Z _I} \psi ^{(\alpha )} _{X _{J} , Z _I} ,
\end{equation}
where $\psi ^{(\alpha )} _{X _{J} , Z _I} \in C ^\infty _c (U ^{(\alpha )} _{X _{J} , Z _I} )$ is defined by
$$ \psi ^{(\alpha )} _{X _{J} , Z _I} (\bx ^{(\alpha )} _{X _{J} , Z _I} (\tau , \mu , x))
:= \theta ^\alpha _{X _{J} , Z _I} (\bx ^{(\alpha )} _{X _{J} , Z _I} (\tau , \mu , x))
\psi (\bx ^{(\alpha )} _{X _{J} , Z _I} (\tau , \mu , P ^\alpha (x)), $$
i.e., by extending some cutoff of $ \psi $ along coordinate curves. We claim that
\begin{prop}
\label{Proof1}
The sum in Equation (\ref{Part}) converges absolutely and
the kernel $\bar \psi \in \Psi ^{- \infty } _{\varepsilon r _\fg ^{-1} C _\rG ; \bzero } (\cG)$.
\end{prop}
\begin{proof}
Given each $U ^\alpha _{X _{J} , Z _I}$, 
consider $d (a , \bs (a))$ and $d (b , \bs (b))$ for any
$a \in U ^\alpha _{X _{J} , Z _I} , b \in U ^\alpha _{X _{J} , Z _I} \bigcap (\cG _1 )$.
By construction, there is a path of length $\leq |I| r _\fg + C' $ joining $a $ and $\bs (a)$,
for some constant $C ' $ independent of $X_{J} , Z _I$.
It follows that $ d (a , \bs (a) ) \leq |I| r _\fg + C' $. 
On the other hand, by (3) of Lemma \ref{GCover}
one has $d (b , \bs (b)) \geq C _\rG |I| - C'' $ for some $C'' > 0 $ independent of $Z _I$.
Rearranging, one gets
$$ d (b , \bs (b) ) \geq \frac{C _\rG  d (a , \bs (a) ) - C'}{r _\fg } - C''.$$

Since by definition, for any $ a \in U ^\alpha _{X _{J} , Z _I}$,
$$ \psi ^{( \alpha )} _{X _{J} , Z _I} (a) = \psi (b) \theta ^\alpha _{X _{J} , Z _I} (a), $$
for some $b \in U ^\alpha _{X _{J} , Z _I} \bigcap \cG _1$,
it follows from our assumption $\psi \in \Psi ^\infty _{\varepsilon , \bzero } (\cG _1 )$ that 
$$ \psi ^{( \alpha )} _{X _{J} , Z _I} (a) \leq M e ^{- \varepsilon ' d (b , \bs (b))}
\leq M ' e^{ - \varepsilon ' C _\rG r _\fg ^{-1} d (a , \bs (a)) } ,$$ 
for some $\varepsilon ' > \varepsilon $.
By the polynomial growth of $\cG$, it follows that 
$\sum _{X _J , Z _I } \psi ^{( \alpha )} _{X _{J} , Z _I} (a) $ 
converges uniformly absolutely and satisfies the estimate
$$ \sum _{X _J , Z _I } \psi ^{( \alpha )} _{X _{J} , Z _I} (a) 
\leq M '' e^{ - \varepsilon ' C _\rG r _\fg ^{-1} d (a , \bs (a)) } ,$$
for some $M'' > 0$.

We turn to the derivatives of $\psi ^{( \alpha )} _{X _{J} , Z _I}$.
Consider $L _{d \tilde \bm (V ^\br \oplus W ^\br )} \psi ^{( \alpha )} _{X _{J} , Z _I} $, 
$V , W \in \Gamma ^\infty (\rA )$.
Write $ V ^\br (\bx ^{(\alpha )} _{X _{J} , Z _I} (\mu , \tau , x ) ) = \sum _l v _l (\mu , \tau , x) \partial _l $.
Then it follows from definition of $\psi ^{( \alpha )} _{X _{J} , Z _I}$ that 
$$ (L _{V ^\br } \psi ^{( \alpha )} _{X _{J} , Z _I}) (\bx ^{(\alpha )} _{X _{J} , Z _I} (\mu , \tau , x ) )
= \sum _l v _l (\mu , \tau , x) (\partial _l \psi )(\tau , \mu , P _\alpha (x) ).$$
Since $v _l (\mu , \tau , x) |V| ^{-1}$ are bounded for all $l$. 
The same arguments for the exponential decay as above can be applied.

We turn to $ L _{d \tilde \bm (0 \oplus W ^\br ) } (\psi ^{( \alpha )} _{X _{J} , Z _I} )$.
We use expression (\ref{XInverse}).
By definition, \\
$\Supp ( \psi ^{( \alpha )} _{X _{J} , Z _I} ) \subseteq \bs ^{-1} (  B (\rM _1 , e ^{- \omega r _\fg |I|} r ))$,
therefore it suffices to consider
$$ L _{d \tilde \bm (0 \oplus W ^\br ) } \psi ^{( \alpha )} _{X _{J} , Z _I} 
(\bx ^{( \alpha )} _{X _{J} , Z _I} (\tau , \mu , x ) ) , 
\quad \Forall W \in \rA |_ { B (\rM _1 , e ^{- \omega r _\fg |I|} r )} .$$
Hence Lemma \ref{VGrowthEst} can be applied to get 
\begin{align*}
L_{( E _{s _1 Y _1 } \circ \cdots \circ E _{Z _1 } W  )^\br } 
(\psi ^{( \alpha )} _{X _{J} , Z _I}) & (\bx ^{(\alpha )} _{X _{J} , Z _I} (\mu , \tau , x ) ) \\
=& e ^{N' (\log |I| )^2} \sum _l w _l (\mu , \tau , x) (\partial _l \psi )(\tau , \mu , P _\alpha (x) ),
\end{align*}
for some functions $w _l $ that are bounded (independent of $I$).

As for 
$d ( \exp (\mu \cdot ( Y _1 , \cdots , Y _p )) 
\exp (\tau \cdot (X ^{\alpha , I} _1 , \cdots , X ^{\alpha , I} _p )) \exp X _{J} \exp Z _I (x)) \nu (W)(\bs (a))$,
write $\nu (W) = \sum _l u _l \partial _{x _l} $ on $U ^{(\alpha )}$.
Then observe that 
\begin{align*}
d \big( \exp (\mu \cdot ( Y _1 , \cdots , Y _p ))
\exp (\tau \cdot (X ^{\alpha , I} _1 , \cdots , X ^{\alpha , I} _p )) & \exp X _{J} \exp Z _I (x) \big) \nu (W) \\
=& \sum _l u _l (x) \partial _{x _l } \quad \text{ on } U ^\alpha _{X _{J} , Z _I}.
\end{align*}
Differentiating $\psi ^{( \alpha )} _{X _{J} , Z _I}$, we get
\begin{align*}
\sum _l u _l (x) \partial _{ x _l } \psi ^{( \alpha )} _{X _{J} , Z _I} & (\mu , \tau , x ) \\
=& \theta ^\alpha _{X _{J}} \theta ^\rG _{Z _I} \psi (\bx ^{(\alpha )} _{X _{J} , Z _I} (\tau , \mu , P ^\alpha (x)) 
\sum _l u _l (x) \partial _{x _l } \theta (2 e ^{\omega r _\fg |I|} r ^{-1} \rho (x )).
\end{align*}
By Lemma \ref{LocalDegen}, and the observation that $x \in B (\rM _1 , e ^{- \omega r _\fg |I|} r )$,
it follows that \\
$\sum _l u _l (x) \frac{\partial }{\partial x _l } \theta (2 e ^{\omega r _\fg |I|} r ^{-1} \rho (x ))$
is also bounded.
Hence we conclude that
$$ L _{d \tilde \bm (V ^\br \oplus W ^\br)}
\psi ^{( \alpha )} _{X _{J} , Z _I} (a) \leq M _1 e^{ - \varepsilon '' C _\rG r _\fg ^{-1} d (a , \bs (a)) } 
(|V | + |W|),$$
for some $\varepsilon '' > \varepsilon $, 
and similar estimate holds for all derivatives. 
Therefore $\bar \psi \in \Psi ^{- \infty } _{\varepsilon r _\fg ^{-1} C _\rG } (\cG)$.  
\end{proof}

Finally, we prove that
\begin{prop}
\label{Proof2}
Suppose $\cG$ is uniformly degenerate.
For any $\kappa \in \Psi ^{- \infty } (\cG ) $ and differential operator $D \in \Psi ^{[m]} (\cG )$,
such that $ D |_{\cG _k } \bar \psi = \kappa | _{\cG _k} $,
$$ D \psi - \kappa \in \Psi ^{- \infty } _{\varepsilon ' ; \lambda } (\cG ).$$
\end{prop}
\begin{proof}
Replacing $\kappa $ by an extension of $\kappa |_{\rM _1 }$ similar to $\bar \psi $,
we may without loss of generality assume $\kappa = 0 $.

Consider $\partial _{x _i} L _{d \tilde \bm (V ^\br \oplus 0 )}\psi ^{(\alpha )} _{X _J , Z _I} $ 
on $U ^\alpha _{X _{J} , Z _I}$.
Recall that $\bx ^\alpha _{\emptyset} ( \mu ' , \tau ' , x ' )$ is a local coordinates chart around 
$ \bt \circ E ^\nu _{X _J} E ^\nu _{Z _I} ( U ^\alpha ) $. Write 
$ V ^\br = \sum _{l = 1 } ^n v _l (x') \partial _{(\mu' , \tau ') _l}$
on $ U ^{(\alpha )} _{\emptyset}$.
Then 
$ V ^\br (\bx ^{(\alpha )} _{X _J , Z _I} ( \mu , \tau , x )) 
= \sum _{l = 1} ^n u _l ( \mu , \tau , E ^\nu _{X _J} E ^\nu _{Z _I } (x)) \partial _{(\mu , \tau ) _l}.$
A straightforward calculation gives
\begin{align*} 
\partial _{x _i} L _{d \tilde \bm (V ^\br \oplus 0 )}& \psi ^{(\alpha )} _{X _J , Z _I} (\mu , \tau , x) \\
=& \theta (2 e ^{\omega r _\fg |I|} r ^{-1} \rho (x ))
\sum _{l = 1} ^n \partial _{x _i} ( u _l ( \mu , \tau ,  E ^\nu _{X _J} E ^\nu _{Z _I } (x))
(\partial _{(\mu , \tau ) _l} \theta ^\alpha _{X _J} \theta ^\rG _{Z _I} ) \psi \\
&+ \big( \partial _{x _i} \theta (2 e ^{\omega r _\fg |I|} r ^{-1} \rho (x )) \big)
\sum _{l = 1} ^n ( u _l ( \mu , \tau ,  E ^\nu _{X _J} E ^\nu _{Z _I } (x))
(\partial _{(\mu , \tau ) _l} \theta ^\alpha _{X _J} \theta ^\rG _{Z _I} ) \psi .
\end{align*}
Since $ E _{Z } $ equals identity on $\rM _ 1 $ for all $Z \in \Span _\bbR \{ Y _1 , \cdots , Y _p \}$,
it follows that there exists some constants $M > 0 $ such that for all $Z \in \Span _\bbR \{ Y _1 , \cdots , Y _p \}$,
$|Z | \leq r _\fg $,
$$| d E ^\nu _{Z } X | \leq ( 1 + M d (x , \rM _1 ))|X| , \Forall X \in T _x M .$$
Iterating, one gets
\begin{equation}
\label{DecayChain}
| d E ^\nu _{Z _I } \partial _{x _i} (x) | 
\leq e ^{\sum _{i = 1} ^{|I| } \log (1 + M e ^{- \omega r _\fg ( |I| - i)} r )} |\partial _{x _i} (x)|.
\end{equation} 
It is elementary that $\sum _{i = 1} ^{|I| } \log (1 + M e ^{- \omega r _\fg ( |I| - i)} r)$ converges.
It follows by integrating 
$\partial _{x _i} L _{d \tilde \bm (V ^\br \oplus 0 )} \psi ^{(\alpha )} _{X _J , Z _I} (\mu , \tau , x)$
with respect to $x _i $ that 
\begin{align*}
L _{d \tilde \bm (V ^\br \oplus 0 )} \psi ^{(\alpha )} _{X _J , Z _I} (\mu , \tau , x)
\leq & L _{d \tilde \bm (V ^\br \oplus 0 )} \psi ^{(\alpha )} _{X _J , Z _I} (\mu , \tau , P ^{(\alpha )} (x)) 
+ M' e ^{- \varepsilon ' d (\bx ^{(\alpha )} _{X _J , Z _I} (\mu , \tau , x ) , x)}\\
& \times ( \rho (\bt (\bx ^{(\alpha )} _{X _J , Z _I} (\mu , \tau , x ))) 
+ e ^{\omega r _\fg |I|} \rho (\bs (\bx ^{(\alpha )} _{X _J , Z _I} (\mu , \tau , x ))) ) ,
\end{align*}
for some constant $M'$.
 
The case $\partial _{x _i} L _{d \tilde \bm (0 \oplus W ^\br)} \psi ^{(\alpha )} _{X _J , Z _I} (\mu , \tau , x) $
is similar.
Again write
\begin{align*}
d \big( \exp (\mu \cdot ( Y _1 , \cdots , Y _p ))
\exp (\tau \cdot (X ^{\alpha , I} _1 , \cdots , X ^{\alpha , I} _p )) & \exp X _{J} \exp Z _I (x) \big) \nu (W) \\
=& \sum _l u _l (x) \partial _{x _l } \quad \text{ on } U ^\alpha _{X _{J} , Z _I}.
\end{align*}
Then one has
\begin{align*}
\partial _{x _i} & \sum _l u _l (x) \partial _{x _l } \psi ^{(\alpha )} _{X _J , Z _I} (\mu , \tau , x) \\
=& \psi (\bx ^{(\alpha )} _{X _{J} , Z _I} (\tau , \mu , P ^\alpha (x)) 
\partial _{x _i} \Big( \theta ^\alpha _{X _{J}} \theta ^\rG _{Z _I} 
\sum _l u _l (x) \partial _{x _l } \theta (2 e ^{\omega r _\fg |I|} r ^{-1} \rho (x )) \Big) .
\end{align*}
It is clear that 
$\partial _{x _i} \big( \theta ^\alpha _{X _{J}} \theta ^\rG _{Z _I} 
\sum _l u _l (x) \partial _{x _l } \theta (2 e ^{\omega r _\fg |I|} r ^{-1} \rho (x )) \big)$
is bounded.

As for $\partial _{x _i} E _{X _J} E _{Z _I} W $,
for each $Z \in \Span _\bbR \{ Y _1 , \cdots , Y _p \}$, write 
$$E _Z \partial _{(\mu ' , \tau ') _l} (x ') 
:= \sum _{l l'} f ^Z _{l l'} (E ^\nu _Z (x ')) \partial _{(\mu ' , \tau ') _l'} (E ^\nu _Z (x ')),$$
for some smooth functions $f ^Z _{l l'}$.
Then one can express $E _{Z _I} W $ as
\begin{align*}
(E _{Z _I} W )^\br ( \bx ^{(\alpha )} _{\emptyset} (\mu ' , \tau ', x ')) =
\sum _{l, l _1 , l _2 , \cdots l _{|I|}, l'} &
f ^{Z _{|I|}} _{l l _{|I|}} (E ^\nu _{Z _{|I|}} \cdots E ^\nu _{Z _1} (x') ) 
\times \cdots \times f ^{Z _1 } _{l l _1} ( E ^\nu _{Z _1} (x') ) \\
& \times w _l (x) \partial _{(\mu' , \tau ') _{l'}} (E ^\nu _{Z _{|l|}} \cdots E ^\nu _{Z _1} (x') ), 
\end{align*}
where $ W ^\br = \sum _{l = 1 } ^n w _l (x') \partial _{(\mu' , \tau ') _l}$
and $x' = (E ^\nu _{Z _{|l|}} \cdots E ^\nu _{Z _1} )^{-1} (x)$ .
Differentiating with respect to $x _i $ and using the estimates (\ref{DecayChain}) and Lemma \ref{VGrowthEst},
one again obtains    
\begin{align*}
L _{d \tilde \bm (0 \oplus W ^\br )} \psi ^{(\alpha )} _{X _J , Z _I} (\mu , \tau , x)
\leq & L _{d \tilde \bm (0 \oplus W ^\br )} \psi ^{(\alpha )} _{X _J , Z _I} (\mu , \tau , P ^{(\alpha )} (x)) 
+ M'' e ^{- \varepsilon ' d (\bx ^{(\alpha )} _{X _J , Z _I} (\mu , \tau , x ) , x)}\\
& \times |I| e ^{(\log |I|)^2} e ^{\omega r _\fg |I|} \rho (\bs (\bx ^{(\alpha )} _{X _J , Z _I} (\mu , \tau , x ))) ,
\end{align*}
for some constant $M''$.

Clearly the same arguments applies for all higher derivatives and one gets similar estimates.
Since $\cG$ is uniformly degenerate, by choosing $U ^\alpha $ to be sufficiently small, 
$\omega $ can be made sufficiently small. 
Hence one can sum over all $X _J , Z _I $ and conclude that
$$ L _{d \tilde \bm (V ^\br \oplus W ^\br )} D \psi \in \Psi ^{- \infty } _{\bullet , 1 } (\cG ),$$
for any differential operators $D $.
\end{proof}
%

\end{document}